\documentclass[12pt]{amsart}
\usepackage{a4wide,enumerate,xcolor}
\usepackage{amsmath,amsthm}
\usepackage{graphicx}
\allowdisplaybreaks
\usepackage{multirow}
\usepackage{caption}

\usepackage{hyperref}
\hypersetup{colorlinks=false}

\let\pa\partial

\let\eps\varepsilon
\newcommand{\N}{{\mathbb N}}
\newcommand{\R}{{\mathbb R}}
\newcommand{\Z}{{\mathbb Z}}

\newcommand{\T}{{\mathcal T}}

\newcommand{\m}{\operatorname{m}}

\newcommand{\torus}{{\mathbb T}}

\newtheorem{theorem}{Theorem}
\newtheorem{lemma}[theorem]{Lemma}

\newtheorem{remark}[theorem]{Remark}

\begin{document}

\title[A convergent finite-volume scheme]{A convergent finite-volume scheme for \\
nonlocal cross-diffusion systems \\
for multi-species populations}

\author[A. J\"ungel]{Ansgar J\"ungel}
\address{Institute of Analysis and Scientific Computing, Technische Universit\"at Wien,
Wiedner Hauptstra\ss e 8--10, 1040 Wien, Austria}
\email{juengel@tuwien.ac.at}

\author[S. Portisch]{Stefan Portisch}
\address{Institute of Analysis and Scientific Computing, Technische Universit\"at Wien, 
Wiedner Hauptstra\ss e 8--10, 1040 Wien, Austria}
\email{stefan.portisch@asc.tuwien.ac.at} 

\author[A. Zurek]{Antoine Zurek}
\address{Universit\'e de Technologie de Compi\`egne, LMAC, 60200 Compi\`egne, France}
\email{antoine.zurek@utc.fr}

\date{\today}

\thanks{The authors have been supported by the Austrian-French Amad\'ee project 
FR 01/2021 of the Austrian Exchange Service (OeAD). The first and second authors acknowledge partial support from
the Austrian Science Fund (FWF), grants P33010 and F65.
This work has received funding from the European
Research Council (ERC) under the European Union's Horizon 2020 research and
innovation programme, ERC Advanced Grant no.~101018153. Finally, the authors thank Maxime Herda for fruitful discussions.}

\begin{abstract}
An implicit Euler finite-volume scheme for a nonlocal cross-diffusion system on the one-dimensional torus, arising in population dynamics,
is proposed and analyzed. The kernels are assumed to be in detailed balance and satisfy a weak cross-diffusion condition. The latter condition allows for negative off-diagonal coefficients and for kernels defined by an indicator function.
The scheme preserves the nonnegativity of the densities,
conservation of mass, and production of the Boltzmann and Rao entropies. 
The key idea is to ``translate'' the entropy calculations for the continuous
equations to the finite-volume scheme, in particular to design 
discretizations of the mobilities, which guarantee a discrete
chain rule even in the presence of nonlocal terms. Based on this idea,
the existence of finite-volume solutions and the convergence of
the scheme are proven. As a by-product, we deduce the existence of weak solutions to the continuous cross-diffusion system. Finally, we present some numerical experiments illustrating the behavior of the solutions to the nonlocal and associated local models.
\end{abstract}



\maketitle


\section{Introduction}

This paper is devoted to the design and analysis of structure-preserving 
finite-volume discretizations of the following one-dimensional nonlocal cross-diffusion initial-value problem:
\begin{align}\label{1.eq}
  & \pa_t u_i = \pa_x(\sigma\pa_x u_i + u_i\pa_x p_i(u)) \quad\mbox{in }\torus,
	\ t>0, \\
	& u_i(\cdot,0) = u_i^0\quad\mbox{in }\torus,\ i=1,\ldots,n, \label{1.ic}
\end{align}
where $\sigma\ge 0$ is the diffusion coefficient, $\torus := \R/\Z$ is the one-dimensional 
torus of unit measure, and $p_i$ is the nonlocal operator
\begin{equation}\label{1.p}
  p_i(u)(x) := a_{ii} u_i(x) + \sum_{\substack{j=1\\j\neq i}}^n a_{ij} (B^{ij}*u_j)(x) 
	= a_{ii} u_i(x)+\sum_{\substack{j=1\\j\neq i}}^n\int_\torus a_{ij} B^{ij}(x-y)u_j(y) dy,
\end{equation}
where $a_{ij}$ are some constants. The kernel functions $B^{ij}:\torus\to\R$ are periodically extended to $\R$,
and $u=(u_1,\ldots,u_n)$ is the solution vector. If we define $B^{ii}=\delta_0$, where $i\in\{1,\ldots,n\}$ and $\delta_0$ is the
Dirac measure, we can rewrite $p_i$ as
\begin{align}\label{1.fullp}
  p_i(u) = \sum_{j=1}^n a_{ij} (B^{ij} * u_j)(x).
\end{align}
Equations \eqref{1.eq} with definition \eqref{1.fullp} 
and general kernels $B^{ij}$ for $i,j=1,\ldots,n$ can be derived
from stochastic interacting particle systems in the many-particle limit \cite{CDJ19}.

We proved in \cite{JPZ22} that the ``full'' nonlocal system, i.e.\ system \eqref{1.eq} and \eqref{1.fullp}, where $B^{ii}\neq \delta_0$ are general kernels, admits global weak solutions. Our analysis was based on the fact that this system possesses two Lyapunov functionals. More precisely, assume that there exist numbers $\pi_1,\ldots,\pi_n>0$ such that the kernels $B^{ij}$ satisfy the so-called detailed-balance condition
\begin{equation*}
  \pi_i a_{ij} B^{ij}(x-y) = \pi_j a_{ji} B^{ji}(y-x) \quad\mbox{for }i,j=1,\ldots,n
	\mbox{ and a.e. }x,y\in\torus,
\end{equation*}
and the positive semi-definiteness condition
\begin{equation}\label{1.pd}
  \sum_{i,j=1}^n\int_\torus \int_{\torus} \pi_i a_{ij} B^{ij}(x-y) v_j(y) 
	v_i(x) dydx \ge 0\quad\mbox{for all }v_i,v_j\in L^2(\torus).
\end{equation} 
Then we proved that the Boltzmann (type) and Rao (type) entropies, respectively,
\begin{align*}
  H_B(u) &= \sum_{i=1}^n\int_\torus\pi_i u_i(\log u_i-1) dx, \\
	H_R(u) &= \frac12\sum_{i,j=1}^n \int_\torus\int_{\torus}\pi_i a_{ij} B^{ij}(x-y) u_j(y) 
	u_i(x) \, dydx,
\end{align*}
fulfill the following entropy dissipation inequalities:
\begin{align}
  \frac{dH_B}{dt} + 4\sigma\sum_{i=1}^n\int_\torus\pi_i|\pa_x\sqrt{u_i}|^2 dx
	&\le -\sum_{i,j=1}^n \int_\torus \int_{\torus}\pi_i a_{ij} B^{ij}(x-y) \pa_x u_j(y) 
	\pa_x u_i(x) dydx, \label{1.HB} \\
	\frac{dH_R}{dt} + \sum_{i=1}^n\int_\torus\pi_i u_i|\pa_x p_i(u)|^2 dx
	&\le -\sigma\sum_{i,j=1}^n \int_\torus\int_{\torus}\pi_i a_{ij}B^{ij}(x-y)\pa_x u_j(y) 
	\pa_x u_i(x) dydx, \label{1.HR}
\end{align}
and the right-hand sides are nonpositive due to \eqref{1.pd}. The Boltzmann entropy is related to the thermodynamic entropy of the system,
and the Rao entropy is a measure of the functional diversity of the species \cite{Rao82}.

While this theoretical framework was suitable to prove the existence of weak solutions, condition \eqref{1.pd} is cumbersome to check in practice. In \cite[Remark 1]{JPZ22}, we proved that \eqref{1.pd} is satisfied for smooth kernels like the Gaussian one, i.e.\ $B^{ij}(x-y)=\exp(-(x-y)^2/2)$ for $i,j=1,\ldots,n$. We also claimed that kernels $B^{ij}$ of the type $B^{ij}=\mathrm{1}_K$ for some interval $K$ around the origin satisfies \eqref{1.pd}. This claim is in fact not true, see the counterexample in Appendix \ref{sec.counter}.

System \eqref{1.eq} and \eqref{1.fullp}, with local or nonlocal self-diffusion terms, describes the dynamics of a population with $n$ species, where the evolution of each species is driven by nonlocal sensing \cite{PoLe19}. In other words, each species has the capability to detect other species over a spatial neighborhood, specified by the kernel $B^{ij}$, and weighted by the strength of attraction ($a_{ij}<0$) or repulsion ($a_{ij}>0$). Thus, from a modeling point of view, the case $B^{ij}=\mathrm{1}_K$ is biologically meaningful. To include this case in our analysis (at the continuous or discrete level), we propose to slightly modify the model studied in \cite{JPZ22} by considering \eqref{1.p} instead of
\eqref{1.fullp}.

For model \eqref{1.eq}--\eqref{1.p}, we impose the following assumptions. We assume that there exist numbers $\pi_1, \ldots, \pi_n > 0$ such that $\pi_i a_{ij} = \pi_j a_{ji}$ for $i,j \in \lbrace 1,\ldots,n \rbrace$, that $B^{ji}(-x)=B^{ij}(x) \geq 0$ for a.e.\ $x \in \torus$ and $i,j \in \lbrace 1,\ldots,n \rbrace$ (with $i \neq j$), and that for all $i,j \in \lbrace 1,\ldots, n \rbrace$ with $i < j$, the matrices
\begin{align}\label{1.M}
  M^{ij}(x) :=
  \begin{pmatrix}
  \pi_i a_{ii} & (n-1)\pi_i a_{ij} B^{ij}(x) \\
  (n-1)\pi_j a_{ji} B^{ij}(x) & \pi_j a_{jj}
  \end{pmatrix}
\end{align}
are positive definite for a.e.\ $x \in \torus$. 
In particular, we could choose some nonpositive off-diagonal coefficients. The possibility to analyze system \eqref{1.eq}--\eqref{1.p} with nonpositive off-diagonal coefficients is a new 
and meaningful result. However, we notice that with these assumptions, the system is only ``weakly'' nonlocal, in the sense that the self-diffusion coefficients have to dominate the 
cross-diffusion terms.

We claim that the functionals $H_B$ and $H_R$ are still entropies for system 
\eqref{1.eq}--\eqref{1.p}, where of course now
\begin{align*}
  H_R(u) = \frac12 \sum_{i=1}^n \int_\torus \pi_i a_{ii} |u_i(x)|^2dx 
	+ \frac12 \sum_{\substack{i,j=1\\i\neq j}}^n \int_\torus \int_\torus \pi_i a_{ij} 
	B^{ij}(x-y) u_j(y) u_i(x) dydx.
\end{align*}
Both functionals satisfy some entropy dissipation inequalities similar to \eqref{1.HB}--\eqref{1.HR}, where, if $i=j$, the terms on the right-hand side are simply given by the square of the $L^2(\torus)$ norm of $\pa_x u_i$. Under the above-mentioned assumptions, the
entropy production term
\begin{align}\label{1.Q}
  Q := \sum_{i=1}^n \int_\torus \pi_i a_{ii} |\pa_x u_i(x)|^2 dx 
	+ \sum_{\substack{i,j=1\\i\neq j}}^n \int_\torus \int_{\torus}\pi_i a_{ij} B^{ij}(x-y) 
	\pa_x u_j(y)\pa_x u_i(x) dydx
\end{align}
is nonnegative; see Lemma \ref{lem.Q} in Appendix \ref{sec.aux}.
Therefore, at least formally, the functionals $H_B$ and $H_R$ are entropies for system \eqref{1.eq}--\eqref{1.p}. In this work, we will translate this property
to the discrete level by analyzing a two-point flux approximation finite-volume scheme for \eqref{1.eq}--\eqref{1.p}.

In the literature, there are some works dealing with the design and analysis of numerical schemes for nonlocal cross-diffusion systems. The work \cite{CHS18} studies a positivity-preserving one-dimensional finite-volume scheme for \eqref{1.eq} with $n=2$ and additional local cross-diffusion terms, with a focus on segregated steady states, but without any numerical analysis. The convergence of this finite-volume scheme was proved in \cite{CFS20}, still focusing on the two-species model.
 A converging finite-volume scheme for a nonlocal cross-diffusion system
modeling either a food chain of three species or, when dropping the cross-diffusion, being an SIR model, was analyzed in \cite{ABS15,BeSe09}. In both models,
the nonlocality comes from the dependence of the self-diffusion coefficients
on the total mass of the corresponding species.
A structure-preserving finite-volume scheme for the nonlocal Shigesada--Kawasaki--Teramoto system was suggested and analyzed in \cite{HeZu22}. 
We also mention the paper \cite{CCH15} for a second-order finite-volume scheme
for a nonlocal diffusion equation, which preserves the nonnegativity and fulfills
a spatially discrete entropy inequality. Related works include a Galerkin scheme for a nonlocal diffusion equation
with additive noise \cite{MeSi22}, a finite-volume discretization of a nonlocal
L\'evy--Fokker--Planck equation \cite{AHHT21}, and numerical schemes for nonlocal diffusion
equations arising in image processing \cite{Gal21}.
Up to our knowledge, there does not exist any numerical analysis of system
\eqref{1.eq}--\eqref{1.p}.

In this paper, we propose a finite-volume scheme which preserves the structure of equations \eqref{1.eq}--\eqref{1.p}. 
Compared to \cite{CFS20}, we allow for an arbitrary number of species, include linear diffusion $\sigma \geq 0$, and prove the preservation of the discrete Boltzmann and Rao entropies. Since we need the positive definiteness of the matrix $M^{ij}(x)$, self-diffusion is needed in our situation. Compared to \cite{HeZu22}, our equations do not have a Laplacian structure, which was used in \cite{HeZu22} to define the numerical scheme, and we allow for nonpositive off-diagonal coefficients.
Our main results can be sketched as follows (see Section \ref{sec.main} for details):

\begin{itemize}
\item We prove the existence of solutions to the finite-volume scheme, which are
nonnegative componentwise, conserve the discrete mass, and satisfy discrete
versions of the entropy inequalities \eqref{1.HB} and \eqref{1.HR}.
\item We show that the discrete solutions converge to a weak solution to 
\eqref{1.eq}--\eqref{1.p} when the mesh size tends to zero.
As a by-product, this proves the existence of a weak solution to \eqref{1.eq}--\eqref{1.ic}.
\item We illustrate numerically the rate of convergence (in space) in the $L^p$-norm as well as the rate of convergence in different metrics of the solution to the nonlocal system towards the solution of the local one (localization limit). Moreover, we illustrate the segregation phenomenon exhibited by the solutions to \eqref{1.eq}--\eqref{1.p}; see \cite{BGHP85}.
\end{itemize}

The paper is organized as follows. The numerical scheme and our main results
are introduced in Section \ref{sec.scheme}. We prove the existence of discrete
solutions in Section \ref{sec.ex}, while the proof of the convergence of the
scheme is presented in Section \ref{sec.conv}. In Section \ref{sec.num},
numerical experiments are given, Appendix \ref{sec.aux} contains
some auxiliary results, and we show in Appendix \ref{sec.counter}
that indicator kernels generally do not fulfill inequality \eqref{1.pd}.


\section{Notation and numerical scheme}\label{sec.scheme}

\subsection{Notation}\label{sec.not}

A uniform mesh $\T$ of the torus $\torus$ consists of $N$ intervals (or cells) $K_\ell$ of length $\Delta x= 1/N$, given by $K_\ell=(x_{\ell-1/2},x_{\ell+1/2})$ with end points $x_{\ell\pm 1/2}=(\ell \pm 1/2) \Delta x$ and centers $x_\ell= \ell \Delta x$ for $\ell\in G=\Z\setminus N\Z$. For given end time $T>0$, let $N_T\in\N$ and define the time step size
$\Delta t = T/N_T$ and the time steps $t_k=k\Delta t$. A space-time
discretization of $Q_T:=\torus\times(0,T)$ is denoted by $\mathcal{D}$; it
consists of the space discretization $\T$ of $\torus$ and 
the time discretization $(N_T,\Delta t)$ of $(0,T)$.

We introduce some function spaces. The space of
piecewise constant (in space) functions is given by
$$
  \mathcal{V}_\T = \bigg\{v:\torus\to\R: \exists (v_\ell)_{\ell\in G}\subset\R,
	\ v(x) = \sum_{\ell\in G}v_\ell\mathrm{1}_{K_\ell}(x)\bigg\},
$$
where $\mathrm{1}_{K_\ell}$ is the indicator function of $K_\ell$. 
We identify the function $v\in\mathcal{V}_\T$ and the numbers
$(v_\ell)_{\ell\in G}$ by writing $v=(v_\ell)_{\ell\in G}$.
For $q\in[1,\infty)$ and $v\in\mathcal{V}_\T$, 
we introduce the $L^q(\torus)$ norm, the
discrete $W^{1,q}(\torus)$ seminorm, and the discrete $W^{1,q}(\torus)$ norm by,
respectively,
\begin{align*}
	\|v\|^q_{0,q,\T} &= \sum_{\ell \in G}\Delta x|v_\ell|^q,\quad 
	|v|_{1,q,\T}^q = \sum_{\ell \in G}\Delta x \bigg|
	\frac{v_{\ell+1}-v_\ell}{\Delta x}\bigg|^q, \\
	\|v\|_{1,q,\T}^q &= |v|_{1,q,\T}^q + \|v\|_{0,q,\T}^q.
\end{align*}
We also define the $L^\infty(\torus)$ norm by
$\|v\|_{0,\infty,\T} = \max_{\ell \in G}|v_\ell|$.
Note that $\|v\|_{0,q,\T}=\|v\|_{L^q(\torus)}$ for functions $v\in\mathcal{V}_\T$.
We set
$$
  \mathrm{D}_\ell v := \frac{v_{\ell+1}-v_\ell}{\Delta x} \quad\mbox{and}\quad
	\mathrm{D}v := (\mathrm{D}_\ell v)_{\ell\in G}.
$$
We recall the definition of the space $\operatorname{BV}(\torus)$
of functions of bounded variation. A function $v \in L^1(\torus)$ belongs to 
$\operatorname{BV}(\torus)$ if its total variation $\operatorname{TV}(v)$, given by
$$
  \operatorname{TV}(v) = \sup\bigg\{ \int_\torus v(x) \partial_x \phi(x) dx:
	\ \phi \in C^1_0(\torus), \ |\phi(x)|\leq 1\quad\mbox{for all }x\in\torus\bigg\},
$$
is finite. We endow the space $\operatorname{BV}(\torus)$ with the norm 
$$
	\|v\|_{{\rm BV}(\torus)} = \|v\|_{L^1(\torus)} + \operatorname{TV}(v)
	\quad \mbox{for all }v \in \operatorname{BV}(\torus).
$$
In particular, it holds $\|v\|_{{\rm BV}(\torus)} = \|v\|_{1,1,\T}$
for any $v \in \mathcal{V}_\T \cap \operatorname{BV}(\torus)$.

For any given $q\in [1,\infty)$, we associate to these norms a dual norm with 
respect to the $L^2(\torus)$ inner product by
$$
  \|v\|_{-1,q',\T} = \sup\bigg\{\bigg|\int_\torus vw dx\bigg|: w\in\mathcal{V}_\T,\
	\|w\|_{1,q,\T}=1\bigg\},
$$
where $1/q+1/q'=1$. Then the following estimate holds for all $v$, $w\in\mathcal{V}_\T$,
$$
  \bigg|\int_\torus vwdx\bigg| \le \|v\|_{-1,q',\T}\|w\|_{1,q,\T}.
$$

We also need the space of piecewise constant (in time) functions taking values
in $\mathcal{V}_\T$:
$$
  \mathcal{V}_{\mathcal D} = \bigg\{v:\torus\times(0,T]\to\R: \exists
	(v^k)_{k=1,\ldots,N_T},\ v(x,t) = \sum_{k=1}^{N_T}\mathrm{1}_{(t_{k-1},t_k]}(t)
	v^k(x)\bigg\},
$$
and the discrete $L^p(0,T;W^{1,q}(\torus))$ norm
$$
  \bigg(\sum_{k=1}^{N_T}\Delta t\|v^k\|_{1,q,\T}^p\bigg)^{1/p},\quad\mbox{where }
	1\le p,q<\infty,\ v\in\mathcal{V}_{\mathcal D}.
$$


\subsection{Numerical scheme}

The initial datum \eqref{1.ic} is approximated by
\begin{equation}\label{2.ic}
  u_{i,\ell}^0 = \frac{1}{\Delta x}\int_{K_\ell}u_i^0(x)dx\quad
	\mbox{for }\ell\in G,\ i=1,\ldots,n.
\end{equation}
For given $k\in\{1,\ldots,N_T\}$ and $u^{k-1}\in\mathcal{V}_\T^n$, the
values $u^k=(u_{i,\ell}^k)_{i=1,\ldots,n,\,\ell\in G}$ are determined by
the implicit Euler finite-volume scheme
\begin{equation}\label{2.fv1}
  \frac{\Delta x}{\Delta t}(u_{i,\ell}^k-u_{i,\ell}^{k-1}) 
	+ \mathcal{F}_{i,\ell+1/2}^k-\mathcal{F}_{i,\ell-1/2}^k = 0, \quad
	i=1,\ldots,n,\ \ell\in G,
\end{equation}
with the numerical fluxes
\begin{equation}\label{2.fv2}
  \mathcal{F}_{i,\ell+1/2}^k = -\frac{\sigma}{\Delta x}(u_{i,\ell+1}^k-u_{i,\ell}^k)
	- \frac{u_{i,\ell+1/2}^k}{\Delta x}(p_{i,\ell+1}^k-p_{i,\ell}^k),
\end{equation}
where the discrete nonlocal operators are given by
\begin{equation}\label{2.fv3}
  p_{i,\ell}^k = a_{ii} u^k_{i,\ell} + \sum_{\substack{j=1\\j \neq i}}^n\sum_{\ell'\in G}
	\Delta x a_{ij}B^{ij}_{\ell-\ell'} u_{j,\ell'}^k,
	\quad B^{ij}_{\ell-\ell'} = \frac{1}{\Delta x}\int_{K_{\ell-\ell'}}	B^{ij}(y)dy,
\end{equation}
for all $i,j=1,\ldots,n$ and $\ell$, $\ell'\in G$. We show in the proof of Lemma \ref{lem.convp} that
$p_{i,\ell}^k = a_{ii} u_i^k(x_\ell) + \sum_{j\neq i}a_{ij}(B^{ij}*u_j^k)(x_\ell)$
for $\ell\in G$, verifying the consistency of the discretization of $p_{i,\ell}^k$.

The mobility $u_{i,\ell+1/2}^k=\widehat{F}(u_{i,\ell}^k,u_{i,\ell+1}^k)$ 
is assumed to satisfy the following properties for all $u_{i,\ell}$, $u_{i,\ell+1}$:
\begin{itemize}
\item The function $\widehat{F}:[0,\infty)^2\to[0,\infty)$ is continuous and satisfies
$\widehat{F}(u_{i,\ell},u_{i,\ell})=u_{i,\ell}$ as well as
$\min\{u_{i,\ell},u_{i,\ell+1}\}\le\widehat{F}(u_{i,\ell},u_{i,\ell+1})\le\max\{u_{i,\ell},u_{i,\ell+1}\}$.
\item There exists $c_0>0$ such that the following discrete chain rule holds:
\begin{equation}\label{2.prop}
  u_{i,\ell+1/2}(p_{i,\ell+1}-p_{i,\ell})(\log u_{i,\ell+1}-\log u_{i,\ell})
	\ge c_0(p_{i,\ell+1}-p_{i,\ell})(u_{i,\ell+1}-u_{i,\ell}).
\end{equation}
\end{itemize}

\begin{remark}[Examples for mobilities]\rm 
Property \eqref{2.prop} is satisfied if $u_{i,\ell}$ 
(we omit the superindex $k$) is defined by the upwind approximation
\begin{equation}\label{2.upwind}
  u_{i,\ell+1/2} = \begin{cases}
	u_{i,\ell+1} &\quad\mbox{if }p_{i,\ell+1}-p_{i,\ell}\ge 0, \\
	u_{i,\ell} &\quad\mbox{if }p_{i,\ell+1}-p_{i,\ell}<0,
	\end{cases}
\end{equation}
or by the logarithmic mean
\begin{equation}\label{2.logmean}
  u_{i,\ell+1/2} = \begin{cases}
	\displaystyle\frac{u_{i,\ell+1}-u_{i,\ell}}{\log u_{i,\ell+1}-\log u_{i,\ell}}
	&\quad\mbox{if }u_{i,\ell+1}>0,\ u_{i,\ell}>0,\mbox{ and } 
	u_{i,\ell+1}\neq u_{i,\ell}, \\
	u_{i,\ell} &\quad\mbox{if }u_{i,\ell+1}=u_{i,\ell} > 0, \\
	0 &\quad\mbox{else}.
	\end{cases}
\end{equation}
We refer to Lemma \ref{lem.prop} in Appendix \ref{sec.aux} for a proof.
\qed
\end{remark}

\begin{remark}[Symmetry of discrete kernels]\rm\label{rem.symm}
Definition \eqref{2.fv3} of $B_{\ell-\ell'}^{ij}$ is consistent with the discrete analog of
$B^{ji}(-x)=B^{ij}(x)$. Indeed, with the change of variables $y\mapsto -y$,
$$
  B_{-\ell'}^{ji} = \frac{1}{\Delta x}\int_{K_{-\ell'}}B^{ji}(y)dy
	= \frac{1}{\Delta x}\int_{K_{\ell'}}B^{ji}(-y)dy
	= \frac{1}{\Delta x}\int_{K_{\ell'}}B^{ij}(y)dy = B_{\ell'}^{ij}. 
$$
\end{remark}

\begin{remark}[Discrete derivative of the convolution]\rm 
A shift of $\Delta x$ in definition \eqref{2.fv3}
of $B^{ij}_{\ell-\ell'}$ shows that $B^{ij}_{\ell-\ell'}=B^{ij}_{(\ell+1)-(\ell'+1)}$,
which leads to
\begin{align}\label{2.B}
  \sum_{\ell'\in G}(B^{ij}_{(\ell+1)-\ell'}-B^{ij}_{\ell-\ell'}) \,u_{j,\ell'}
	&= \sum_{\ell'\in G}\big(B^{ij}_{(\ell+1)-(\ell'+1)} u_{j,\ell'+1}
	- B^{ij}_{\ell-\ell'}u_{j,\ell'}\big) \\
	&= \sum_{\ell'\in G} B^{ij}_{\ell-\ell'}(u_{j,\ell'+1}-u_{j,\ell'}) \nonumber
\end{align}
for all $\ell\in G$, $i,j=1,\ldots,n$. 
This is the discrete analog of the rule $\pa_x B^{ij}*u_j=B^{ij}*\pa_x u_j$.
\qed
\end{remark}

\begin{remark}[Asymptotic-preserving scheme]\rm\label{rem.ap}
For $j\neq i$, let $B^{ij}=B^{ij}_\eps$ for some parameter $\eps\to 0$ and 
$B^{ij}_\eps\to \delta_0$ in the sense of distributions as $\eps\to 0$. Let $p_{i,\ell}^{k,\eps}$ be defined as in \eqref{2.fv3} with $B^{ij}(y)$ replaced
by $B^{ij}_\eps(y)$. Then, as $\eps \to 0$,
\begin{align*}
p_{i,\ell}^{k,\eps} \to \sum_{j=1}^n a_{ij} \left(\delta_ 0 \ast u_j\right)(x_\ell)= \sum_{j=1}^n a_{ij} u_{j, \ell}.
\end{align*}
Thus, our numerical scheme is asymptotic preserving in the sense that the
method converges to a finite-volume scheme for the local system, which also
preserves the nonnegativity, conserves the mass, and dissipates the Boltzmann
and Rao entropies.
\qed
\end{remark}


\subsection{Main results}\label{sec.main}

We impose the following hypotheses:

\begin{itemize}
\item[(H1)] Domain and parameters: $\torus$ is a one-dimensional torus,
$T>0$, $\sigma\geq 0$, and $Q_T:=\torus\times(0,T)$.
\item[(H2)] Initial datum: $u^0=(u_1^0,\ldots,u_n^0)\in L^2(\torus;\R^n)$ satisfies 
$u_i^0\ge 0$ in $\torus$.
\item[(H3)] Kernels: Let $B^{ij}\in L^\infty(\torus)$ for $j\neq i$ be a nonnegative 
function satisfying $B^{ji}(x) = B^{ij}(-x)$ for a.e.\ $x \in \torus$. There exist numbers 
$\pi_1, \ldots, \pi_n > 0$ such that $\pi_i a_{ij} = \pi_j a_{ji}$ (detailed-balance condition),
and the matrices $M^{ij}$, defined in \eqref{1.M}, are positive definite
for a.e.\ $x\in\torus$.
\end{itemize}

We consider the one-dimensional equations mainly for notational simplicity. In several space dimensions $d>1$, we infer uniform estimates in spaces with
weaker integrability than in one space dimension, because of Sobolev embeddings.
Thanks to the positive definiteness condition on $M^{ij}_{\ell-\ell'}$,
we obtain a bound for $u_i$ in the discrete $L^2(0,T;H^1(\torus))$ norm,
which allows us to conclude, together with the Rao entropy estimate,
by the discrete Gagliardo--Nirenberg inequality,
a bound for $u_i$ in $L^{2+4/d}(Q_T)$, which is sufficient to estimate the product
$u_i\pa_x p_i(u)$. In the one-dimensional situation, this procedure simplifies; 
see Lemma \ref{lem.aux1}.

Our results also hold if $\sigma=0$, since the condition $\sigma>0$ provides
an estimate for $u_i$ in the discrete norm of $L^2(0,T;W^{1,1}(\torus))$, while the
positive definiteness condition on $M^{ij}_{\ell-\ell'}$ allows us to conclude a stronger bound
in the discrete norm of $L^2(0,T;H^1(\torus))$.
Notice that kernels of the type $B^{ij}=\mathrm{1}_K$ satisfy Hypothesis (H3) (for suitable
$\pi_i$ and $a_{ij}$).

Condition $u^0\in L^2(\torus;\R^n)$ in Hypothesis (H2) is needed to obtain a 
finite initial Rao entropy $H_R(u^0)$. 
For the existence result, the assumption on the kernels can be weakened to
$B^{ij}\in L^1(\torus)$. The boundedness condition on $B^{ij}$ in Hypothesis (H3)
is needed in the proof of the convergence of the scheme. 

We introduce for a given nonnegative function $u\in\mathcal{V}_\T^n$ the discrete entropies
\begin{align}
  \mathcal{H}_B(u) &= \sum_{i=1}^n\sum_{\ell\in G}\Delta x\pi_i h(u_{i,\ell}), \quad
	h(s) = s(\log s-1), \label{2.HB} \\
  \mathcal{H}_R(u) &= \frac12 \sum_{i=1}^n \sum_{\ell \in G} \Delta x \pi_i a_{ii} 
	|u_{i,\ell}|^2 + \frac12\sum_{\substack{i,j=1\\i \neq j}}^n\sum_{\ell,\ell' \in G}
	(\Delta x)^2\pi_i a_{ij} B^{ij}_{\ell-\ell'} u_{j,\ell'} u_{i,\ell},  \nonumber 
\end{align}
and the matrices
\begin{align}\label{2.Mijdis}
  M^{ij}_{\ell - \ell'} :=
  \begin{pmatrix}
  \pi_i a_{ii} & (n-1)\pi_i a_{ij} B^{ij}_{\ell-\ell'} \\
  (n-1)\pi_j a_{ji} B^{ij}_{\ell-\ell'} & \pi_j a_{jj}
  \end{pmatrix} \quad\mbox{for }i<j,\ \ell, \ell' \in G.
\end{align}
In view of Hypothesis (H3), they are symmetric and positive definite uniformly
in $\ell,\ell'\in G$, i.e.\
$z^\top M^{ij}_{\ell-\ell'}z \ge c_M|z|^2$ for all $z\in\R^2$ and some $c_M>0$.

Our first main result is the existence of discrete solutions.

\begin{theorem}[Existence of discrete solutions]\label{thm.ex}
Let Hypotheses (H1)--(H3) hold. Then there exists a solution
$u^k\in\mathcal{V}_\T^n$ to \eqref{2.ic}--\eqref{2.fv3} for all $k=1,\ldots,N_T$,
satisfying $u_{i,\ell}^k\ge 0$ for all $i=1,\ldots,n$, $\ell\in G$ and
the discrete entropy inequalities
\begin{align}
  \mathcal{H}_B(u^k) 
	&+ \frac{c_0\Delta t}{n-1}\sum_{\substack{i,j=1\\i<j}}^n
	\sum_{\ell,\ell'\in G}(\Delta x)^2
  \begin{pmatrix} \mathrm{D}_\ell u_i^{k} \\ \mathrm{D}_{\ell'}u_j^{k}
  \end{pmatrix}^\top M^{ij}_{\ell-\ell'}
  \begin{pmatrix} \mathrm{D}_\ell u_i^{k} \\ \mathrm{D}_{\ell'}u_j^{k}
  \end{pmatrix} \label{2.Bei} \\
	&{}+ 4\sigma\Delta t\sum_{i=1}^n\pi_i|(u_i^k)^{1/2}|^2_{1,2,\T}
	\le \mathcal{H}_B(u^{k-1}), \nonumber \\
	\mathcal{H}_R(u^k) &+ \Delta t\sum_{i=1}^n\sum_{\ell\in G}\Delta x\pi_i
	u_{i,\ell+1/2}^k\bigg(\frac{p_{i,\ell+1}^k-p_{i,\ell}^k}{\Delta x}\bigg)^2 
	\label{2.Rei} \\
	&{}+ \frac{\sigma\Delta t}{(n-1)}\sum_{\substack{i,j=1\\i<j}}^n
	\sum_{\ell,\ell'\in G}(\Delta x)^2
  \begin{pmatrix} \mathrm{D}_\ell u_i^{k} \\ \mathrm{D}_{\ell'}u_j^{k}
  \end{pmatrix}^\top M^{ij}_{\ell-\ell'}
  \begin{pmatrix} \mathrm{D}_\ell u_i^{k}\\ \mathrm{D}_{\ell'}u_j^{k}
  \end{pmatrix}
	\le \mathcal{H}_R(u^{k-1}). \nonumber
\end{align}
Furthermore, the solution conserves the mass, $\sum_{\ell\in G}\Delta x u_{i,\ell}^k
= \int_\torus u_i^0(x)dx$ for all $i=1,\ldots,n$, $k=1,\ldots,N_T$.
\end{theorem}

This theorem is proved by solving a fixed-point problem based on a topological
degree argument, similar as in \cite{JuZu21}. For this, we formulate \eqref{2.fv1}
in terms of the entropy variable $w_{i}=\pi_i\log u_{i}$ and regularize
the equations by adding the discrete analog of $-\eps\Delta w_i+\eps w_i$.
The regularization ensures the coercivity in the variable $w_i$. After transforming
back to the original variable $u_{i}=\exp(w_{i}/\pi_i)$, we obtain
automatically the positivity of $u_i$ (and nonnegativity after passing to the
limit $\eps\to 0$). Like on the continuous level, the derivation of the
discrete entropy inequalities \eqref{2.Bei} and \eqref{2.Rei} relies
on the detailed-balance condition $\pi_i a_{ij}=\pi_j a_{ji}$ for all $i,j=1,\ldots,n,$. 

For our second main result, we need to introduce some notation. We define the ``diamond'' 
cell of the dual mesh $T_{\ell+1/2}=(x_\ell,x_{\ell+1})$ with center $x_{\ell+1/2}$. 
These cells define another partition of $\torus$. The gradient of 
$v\in\mathcal{V}_{\mathcal D}$ is then defined by
$$
  \pa_x^{\mathcal D} v(x,t) = \mathrm{D}_\ell v^k
	= \frac{v_{\ell+1}^k-v_\ell^k}{\Delta x}
	\quad\mbox{for }x\in T_{\ell+1/2},\ t\in(t_{k-1},t_k].
$$
We also introduce a sequence of space-time discretizations 
$(\mathcal{D}_m)_{m\in\N}$ indexed by the mesh size $\eta_m=\max\{\Delta x_m,
\Delta t_m\}$ satisfying $\eta_m\to 0$ as $m\to\infty$. The corresponding spatial mesh is denoted by $\T_m$ with $G_m = \Z \setminus N_m \Z$ and the number of time steps by $N_T^m$. Finally, to simplify the notation, we set $\pa_x^m:=\pa_x^{\mathcal{D}_m}$.

\begin{theorem}[Convergence of the scheme]\label{thm.conv}
Let Hypotheses (H1)--(H3) hold and let $\mathcal{D}_m$ be a sequence of uniform
space-time discretizations satisfying $\eta_m\to 0$ as $m\to\infty$. Let
$(u_m)$ be the solutions to \eqref{2.ic}--\eqref{2.fv3} constructed in
Theorem \ref{thm.ex}. Then there exists $u=(u_1,\ldots,u_n)$ satisfying
$u_i\ge 0$ in $Q_T$ and, up to a subsequence, as $m\to\infty$,
\begin{align*}
  u_{i,m}\to u_i &\quad\mbox{strongly in } L^{2}(Q_T), \\
	\pa_x^m u_{i,m}\rightharpoonup \pa_x u_i &\quad\mbox{weakly in } L^{2}(Q_T),
\end{align*}
and $u$ is a weak solution to \eqref{1.eq}--\eqref{1.ic}, i.e., it holds for all
$\psi_i\in C_0^\infty(\torus\times[0,T))$ and $i=1,\ldots,n$ that
$$
  \int_0^T\int_\torus u_i\pa_t\psi_i dxdt + \int_\torus u_i^0\psi_i(\cdot,0)dx
	= \int_0^T\int_\torus(\sigma\pa_x u_i + u_i\pa_x p_i(u))\pa_x\psi_i dxdt.
$$
\end{theorem}

The proof of Theorem \ref{thm.conv}
is based on suitable estimates uniform with respect to $\Delta x_m$ 
and $\Delta t_m$, derived from the discrete entropy inequalites. A discrete
version of the Aubin--Lions lemma from \cite{GaLa12} yields the strong
convergence of a subsequence of $(u_m)$ of solutions to \eqref{2.fv1}--\eqref{2.fv3}.
The most technical part is the identification of the limit function as a
weak solution to \eqref{1.eq}--\eqref{1.ic}.


\section{Proof of Theorem \ref{thm.ex}}\label{sec.ex}

Theorem \ref{thm.ex} is proved by induction over $k=1,\ldots,N_T$. 
We first regularize the problem and prove the existence of an approximate solution
by using a topological degree argument for the fixed-point problem. 
The discrete entropy inequalities yield
a priori estimates independent of the approximation parameter. The
deregularization limit is performed thanks to the Bolzano--Weierstra{\ss} theorem. 

Let $k\in\{1,\ldots,N_T\}$ and $u^{k-1}\in\mathcal{V}_\T^n$ satisfying
$u_{i,\ell}^{k-1}\ge 0$ for $i=1,\ldots,n$, $\ell\in G$ be given.

\subsection{Solution to a linearized regularized scheme}

We prove the existence of a unique solution to a linearized regularized problem,
which allows us to define the fixed-point operator. Let $R>0$, $\eps>0$ and define
$$
  Z_R = \big\{w=(w_1,\ldots,w_n)\in\mathcal{V}_\T^n: \|w_i\|_{1,2,\T}<R\mbox{ for }
	i=1,\ldots,n\big\}.
$$
We introduce the mapping $F:Z_R\to\R^{nN}$, $w\mapsto w^\eps$, where $w^\eps$
is the solution to the linear regularized problem
\begin{equation}\label{3.lin}
  -\eps\frac{w_{i,\ell+1}^\eps - 2w_{i,\ell}^\eps 
	+ w_{i,\ell-1}^\eps}{\Delta x}
	+ \eps\Delta x w_{i,\ell}^\eps 
	= -\Delta x\frac{u_{i,\ell}-u_{i,\ell}^{k-1}}{\Delta t}
	- (\mathcal{F}_{i,\ell+1/2} - \mathcal{F}_{i,\ell-1/2}),
\end{equation}
where $i=1,\ldots,n$, $\ell\in G$, $u_{i,\ell}$ is defined by
$u_{i,\ell}=\exp(w_{i,\ell}/\pi_i)$, $\mathcal{F}_{i,\ell\pm 1/2}$ is defined
as in \eqref{2.fv2} with $u_i^k$ replaced by $u_i$ and $p_{i,\ell}^k$ replaced by
\[
  p_{i,\ell} = a_{ii} u_{i,\ell}+ \sum_{\substack{j=1\\j\neq i}}^n\sum_{\ell'\in G}\Delta x 
	a_{ij} B^{ij}_{\ell-\ell'} u_{j,\ell'}.
\]

We claim that $F$ is well defined. For this, 
we write \eqref{3.lin} in the form
$$
  Mw^\eps = v, \quad\mbox{where }
	v_{i,\ell} = -\Delta x \frac{u_{i,\ell}-u_{i,\ell}^{k-1}}{\Delta t}
	- (\mathcal{F}_{i,\ell+1/2} - \mathcal{F}_{i,\ell-1/2}).
$$
The matrix $M\in\R^{nN\times nN}$ is a block diagonal matrix with entries 
$M'\in\R^{N\times N}$, which are tridiagonal matrices such that
$M'_{\ell,\ell}=\eps\Delta x + 2\eps/\Delta x$, 
$M'_{\ell+1,\ell} = M'_{\ell,\ell+1} = -\eps/\Delta x$.
We can decompose the full system $Mw^\eps=v$ into the subsystems
$M'w_i^\eps=v_i$ for $i=1,\ldots,n$. Since $M'$ is strictly diagonally
dominant, there exists a unique solution to $M'w_i^\eps=v_i$ and consequently
for $Mw^\eps=v$ by setting $w^\eps=(w_1^\eps,\ldots,w_n^\eps)$. We infer that
the mapping $F$ is well defined.

\subsection{Continuity of $F$}\label{sec.cont}

We fix $i\in\{1,\ldots,n\}$, multiply \eqref{3.lin} by $w_{i,\ell}^\eps$, and
sum over $\ell\in G$:
\begin{align}\label{3.aux}
  \eps\sum_{\ell\in G}&\bigg(-\frac{w_{i,\ell+1}^\eps-2w_{i,\ell}^\eps
	+ w_{i,\ell-1}^\eps}{\Delta x} + \Delta x w_{i,\ell}^\eps\bigg)w_{i,\ell}^\eps \\
	&= -\sum_{\ell\in G}\Delta x
	\frac{u_{i,\ell}-u_{i,\ell}^{k-1}}{\Delta t}w_{i,\ell}^\eps
	- \sum_{\ell\in G}(\mathcal{F}_{i,\ell+1/2}-\mathcal{F}_{i,\ell-1/2})w_{i,\ell}^\eps.
	\nonumber
\end{align}
The left-hand side can be rewritten by using discrete integration by parts
(or summation by parts):
\begin{align}\label{3.dibp}
  \eps\sum_{\ell\in G}&\bigg(-\frac{(w_{i,\ell+1}^\eps-w_{i,\ell}^\eps)
	-(w_{i,\ell}^\eps-w_{i,\ell-1}^\eps)}{\Delta x}w_{i,\ell}^\eps 
	+ \Delta x(w_{i,\ell}^\eps)^2\bigg) \\
	&= \eps\sum_{\ell\in G}\frac{(w_{i,\ell+1}^\eps-w_{i,\ell}^\eps)^2}{\Delta x}
	+ \eps\sum_{\ell\in G}\Delta x(w_{i,\ell}^\eps)^2 
	= \eps\|w_i^\eps\|_{1,2,\T}^2. \nonumber
\end{align}
The first term on the right-hand side of \eqref{3.aux} is estimated by
the Cauchy--Schwarz inequality, taking into account that $w\in Z_R$, which
implies a finite discrete $L^2(\torus)$ norm for $u_{i,\ell}=\exp(w_{i,\ell}/\pi_i)$:
$$
  \bigg|-\sum_{\ell\in G}\Delta x
	\frac{u_{i,\ell}-u_{i,\ell}^{k-1}}{\Delta t}w_{i,\ell}^\eps\bigg|
	\le C(\Delta t)\|u_i-u_i^{k-1}\|_{0,2,\T}\|w_i^\eps\|_{0,2,\T}
	\le C(\Delta t,R)\|w_i^\eps\|_{1,2,\T},
$$
where here and in the following $C>0$, $C(\Delta t,R)>0$, etc.\ are 
generic constants with values changing from line to line.
We split the second term on the right-hand side of \eqref{3.aux} into two parts:
\begin{align*}
  & -\sum_{\ell\in G}(\mathcal{F}_{i,\ell+1/2}-\mathcal{F}_{i,\ell-1/2})w_{i,\ell}^\eps
	= I_1 + I_2, \quad\mbox{where} \\
	& I_1 = \sigma\sum_{\ell\in G}\bigg(\frac{u_{i,\ell+1}-u_{i,\ell}}{\Delta x}
	- \frac{u_{i,\ell}-u_{i,\ell-1}}{\Delta x}\bigg)w_{i,\ell}^\eps, \\
	& I_2 = \sum_{\ell\in G}\bigg(u_{i,\ell+1/2}\frac{p_{i,\ell+1}-p_{i,\ell}}{\Delta x}
	- u_{i,\ell-1/2}\frac{p_{i,\ell}-p_{i,\ell-1}}{\Delta x}\bigg)w_{i,\ell}^\eps.
\end{align*}
For $I_1$, we use discrete integration by parts, the Cauchy--Schwarz inequality,
and the fact that $w\in Z_R$:
\begin{align*}
  |I_1| &= \bigg|-\sigma\sum_{\ell\in G}\Delta x
	\frac{u_{i,\ell+1}-u_{i,\ell}}{\Delta x}
	\frac{w_{i,\ell+1}^\eps-w_{i,\ell}^\eps}{\Delta x}\bigg| \\
	&\le \sigma\bigg(\sum_{\ell\in G}\Delta x
	\bigg|\frac{u_{i,\ell+1}-u_{i,\ell}}{\Delta x}\bigg|^2\bigg)^{1/2}
	\bigg(\sum_{\ell\in G}\Delta x
	\bigg|\frac{w_{i,\ell+1}^\eps-w_{i,\ell}^\eps}{\Delta x}\bigg|^2\bigg)^{1/2} \\
	&= \sigma|u_i|_{1,2,\T}|w_i^\eps|_{1,2,\T} \le C(R)\|w_i^\eps\|_{1,2,\T}.
\end{align*}
Using discrete integration by parts, and definition \eqref{2.fv3} of $p_{i,\ell}$, we obtain
\begin{align*}
|I_2| &= \bigg|-\sum_{\ell\in G}\Delta x u_{i,\ell+1/2}
	\frac{p_{i,\ell+1}-p_{i,\ell}}{\Delta x}
	\frac{w_{i,\ell+1}^\eps-w_{i,\ell}^\eps}{\Delta x}\bigg|\leq I_{21} + I_{22},
	\quad\mbox{where} \\
  I_{21} &= \bigg| \sum_{\ell \in G} \Delta x u_{i,\ell+1/2}a_{ii} 
	\frac{(u_{i,\ell+1}-u_{i,\ell})}{\Delta x}
	\frac{(w^\eps_{i,\ell+1}- w^\eps_{i,\ell})}{\Delta x}\bigg|, \\
  I_{22} &= \bigg|\sum_{\substack{j=1\\j\neq i}}^n\sum_{\ell,\ell'\in G}(\Delta x)^2 
	u_{i,\ell+1/2}a_{ij}	\frac{B^{ij}_{\ell+1-\ell'}-B^{ij}_{\ell-\ell'}}{\Delta x}u_{j,\ell'}
	\frac{w_{i,\ell+1}^\eps-w_{i,\ell}^\eps}{\Delta x}\bigg|.
\end{align*}
For $I_{21}$, because of the bound in $Z_R$, 
we can estimate $u_{i,\ell+1/2}\le\max\{u_{i,\ell+1},u_{i,\ell}\}\le C(R)$. Then, thanks to the Cauchy--Schwarz inequality, we obtain
\begin{align*}
I_{21} \leq C(R) a_{ii} \, |u_i|_{1,2,\T} \, |w^\eps_i|_{1,2,\T} \leq C(R) \, \|w^\eps_i\|_{1,2,\T}.
\end{align*}
For $I_{22}$, applying the discrete analog \eqref{2.B} of the rule 
$\pa_x B^{ij}*u_j=B^{ij}*\pa_x u_j$, 
\begin{align*}
  I_{22} &= \bigg|\sum_{\substack{j=1\\j\neq i}}^n\sum_{\ell,\ell'\in G}(\Delta x)^2 
	u_{i,\ell+1/2} a_{ij}	B^{ij}_{\ell-\ell'}\frac{u_{j,\ell'+1}-u_{j,\ell'}}{\Delta x}
	\frac{w_{i,\ell+1}^\eps-w_{i,\ell}^\eps}{\Delta x}\bigg| \\
	&= \bigg|\sum_{\substack{j=1\\j\neq i}}^n\sum_{\ell,\ell'\in G}(\Delta x)^2 
	u_{i,\ell+1/2} a_{ij}	B^{ij}_{\ell-\ell'}(\mathrm{D}_{\ell'}u_j)(\mathrm{D}_\ell w_i)\bigg|,
\end{align*}
where we used the notation of Section \ref{sec.not}. Similarly to $I_{21}$, we infer that
\begin{align*}
  I_{22} \leq C(R) \sum_{\substack{j=1\\j\neq i}}^n a_{ij} \sum_{\ell \in G} \Delta x 
	\bigg(\sum_{\ell' \in G} \Delta x B^{ij}_{\ell-\ell'}\mathrm{D}_{\ell'} u_j \bigg)
	\mathrm{D}_\ell w_i.
\end{align*}
Then, by the Cauchy--Schwarz inequality
and the discrete convolution inequality from Lemma \ref{lem.young} in
Appendix \ref{sec.aux}, 
\begin{align*}
  I_{22} &\le C(R)\sum_{\substack{j=1\\j\neq i}}^n \bigg\{\sum_{\ell \in G} \Delta x 
	\bigg(\sum_{\ell' \in G} \Delta x B^{ij}_{\ell-\ell'}\mathrm{D}_{\ell'}u_j\bigg)^2 
	\bigg\}^{1/2}|w_i|_{1,2,\T} \\
	&\le C(R)\sum_{\substack{j=1\\j\neq i}}^n\|B^{ij}\|_{L^1(\torus)}|u_j|_{1,2,\T}|w_i|_{1,2,\T}
	\le C(R)\|w_i\|_{1,2,\T}.
\end{align*}
Combining these estimates, we deduce from \eqref{3.aux} that
$\eps\|w_i^\eps\|_{1,2,\T}\le C(\Delta t,R)$. 

We can proceed to show the continuity of $F$. Let $(w^k)_{k\in\N}$ be such that
$w^k\to w\in Z_R$ as $k\to\infty$ and set $w^{\eps,k}:=F(w^k)$. We have just proved
that $(w^{\eps,k})_{k\in\N}$ is bounded with respect to the $\|\cdot\|_{1,2,\T}$
norm. By the Bolzano--Weierstra{\ss} 
theorem, there exists a subsequence (not relabeled) such that $w^{\eps,k}\to w^\eps$
in $Z_R$ as $k\to\infty$. Performing the limit $k\to\infty$ in \eqref{3.aux}, 
satisfied for $w^{\eps,k}$, shows that $w^\eps$ solves scheme \eqref{3.aux} with
$u_{i,\ell}=\exp(w_i^\eps/\pi_i)$. This means that $w^\eps=F(w)$ and proves the
continuity of $F$.

\subsection{Existence of a fixed point}\label{sec.fixed}

We show that $F:Z_R\to\R^{nN}$ admits a fixed point by using a topological degree
argument. We recall that the Brouwer topological degree is a mapping 
$\operatorname{deg}:M\to\mathbb{Z}$, where
$$
  M = \big\{(f,Z,y):f\in C^0(\torus),\ Z\mbox{ is open, bounded},
	\ y\not\in f(\pa Z)\big\};
$$
see \cite[Chap.~1, Theorem~3.1]{Dei85} for details and properties.
If we show that any solution $(w^\eps,\rho)\in\overline{Z}_R\times[0,1]$ to the
fixed-point equation $w^\eps = \rho F(w^\eps)$ satisfies 
$(w^\eps,\rho)\not\in \pa Z_R\times[0,1]$ for sufficiently large values of $R>0$,
then we deduce from the invariance by homotopy that
$\operatorname{deg}(I-\rho F,Z_R,0)$ is invariant in $\rho$. Then, choosing
$\rho=0$, $\operatorname{deg}(I,Z_R,0)=1$ and, if $\rho=1$,
$\operatorname{deg}(I-F,Z_R,0)=\operatorname{deg}(I,Z_R,0)=1$.
This implies that there exists $w^\eps\in Z_R$ such that $(I-F)(w^\eps)=0$,
which is the desired fixed point.
 
Let $(w^\eps,\rho)$ be a fixed point of $w^\eps=\rho F(w^\eps)$. 
If $\rho=0$, there is nothing to show.
Therefore, let $\rho>0$. Then $w_i^\eps$ solves
\begin{equation}\label{3.weps}
  -\eps\frac{w_{i,\ell+1}^\eps-2w_{i,\ell}^\eps+w_{i,\ell-1}^\eps}{\Delta x}
	+ \eps\Delta x w_{i,\ell}^\eps = -\rho\bigg(\Delta x
	\frac{u_{i,\ell}^\eps-u_{i,\ell}^{k-1}}{\Delta t} + \mathcal{F}_{i,\ell+1/2}^\eps
	- \mathcal{F}_{i,\ell-1/2}^\eps\bigg)
\end{equation}
for all $\ell\in G$ and $i=1,\ldots,n$, where 
$u_{i,\ell}^\eps = \exp(w_{i,\ell}^\eps/\pi_i)$, and the fluxes 
$\mathcal{F}_{i,\ell\pm 1/2}^\eps$ are defined as in \eqref{2.fv2}
with $u_{i,\ell}^k$ replaced by $u_{i,\ell}^\eps$. We multiply the previous equation
by $\Delta t w_{i,\ell}^\eps$, sum over $\ell\in G$, $i=1,\ldots,n$, and
use discrete integration by parts as in \eqref{3.dibp}:
\begin{equation}\label{3.aux2}
  \eps\Delta t\sum_{i=1}^n\|w_i^\eps\|_{1,2,\T}^2
	= -\rho\sum_{i=1}^n\sum_{\ell\in G}\big(\Delta x(u_{i,\ell}^\eps-u_{i,\ell}^{k-1})
	w_{i,\ell}^\eps 
	+ \Delta t(\mathcal{F}_{i,\ell+1/2}^\eps-\mathcal{F}_{i,\ell-1/2}^\eps)
	w_{i,\ell}^{\eps}\big).
\end{equation}
For the first term on the right-hand side, we use $w_{i,\ell}^\eps
= \pi_i\log u_{i,\ell}^\eps$ and the convexity of $h(s)=s(\log s-1)$:
$$
  (u_{i,\ell}^\eps-u_{i,\ell}^{k-1})\pi_i\log u_{i,\ell}^\eps
  \ge \pi_i\big(h(u_{i,\ell}^\eps)-h(u_{i,\ell}^{k-1})\big).
$$
Recalling definition \eqref{2.HB} of $\mathcal{H}_B$, this shows that
$$
  -\rho\sum_{i=1}^n\sum_{\ell\in G}\Delta x(u_{i,\ell}^\eps-u_{i,\ell}^{k-1})
	w_{i,\ell}^\eps \le -\rho\big(\mathcal{H}_B(u^\eps) - \mathcal{H}_B(u^{k-1})\big).
$$
Like in Section \ref{sec.cont}, we split the second term in \eqref{3.aux2}
into two parts:
\begin{align}\label{3.I34}
  & -\rho\Delta t\sum_{i=1}^n\sum_{\ell\in G}
	(\mathcal{F}_{i,\ell+1/2}^\eps-\mathcal{F}_{i,\ell-1/2}^\eps)w_{i,\ell}^{\eps}
	= I_3 + I_4, \quad\mbox{where} \\
	& I_3 = \rho\sigma\Delta t\sum_{i=1}^n\sum_{\ell\in G}\bigg(
	\frac{u_{i,\ell+1}^\eps-u_{i,\ell}^\eps}{\Delta x} 
	- \frac{u_{i,\ell}^\eps-u_{i,\ell-1}^\eps}{\Delta x}\bigg)
	w_{i,\ell}^\eps, \nonumber \\
	& I_4 = \rho\Delta t\sum_{i=1}^n\sum_{\ell\in G}\bigg(u_{i,\ell+1/2}^\eps
	\frac{p_{i,\ell+1}^\eps-p_{i,\ell}^\eps}{\Delta x} - u_{i,\ell-1/2}^\eps
	\frac{p_{i,\ell}^\eps-p_{i,\ell-1}^\eps}{\Delta x}\bigg)w_{i,\ell}^\eps. \nonumber
\end{align}
We use discrete integration by parts, the definition
$w_{i,\ell}^\eps = \pi_i \log u_{i,\ell}^\eps$, and the elementary inequality
$(a-b)(\log a-\log b)\ge 4(\sqrt{a}-\sqrt{b})^2$ for $a$, $b>0$ to estimate
the first term:
\begin{align*}
  I_3 &= -\rho\sigma\Delta t\sum_{i=1}^n\sum_{\ell\in G}
	\frac{u_{i,\ell+1}^\eps-u_{i,\ell}^\eps}{\Delta x}
	(w_{i,\ell+1}^\eps-w_{i,\ell}^\eps) \\
	&\le -4\rho\sigma\Delta t\sum_{i=1}^n\sum_{\ell\in G}\frac{\pi_i}{\Delta x}
	\big((u_{i,\ell+1}^\eps)^{1/2}-(u_{i,\ell}^\eps)^{1/2}\big)^2 
	= -4\rho\sigma\Delta t\sum_{i=1}^n\pi_i|(u_i^\eps)^{1/2}|_{1,2,\T}^2.
\end{align*}

For the second term $I_4$, we use discrete integration by parts and
$w_{i,\ell}^\eps = \pi_i \log u_{i,\ell}^\eps$ again as well as
property \eqref{2.prop} (discrete chain rule):
\begin{align*}
  I_4 &= -\rho\frac{\Delta t}{\Delta x}\sum_{i=1}^n\sum_{\ell\in G}\pi_i
	u_{i,\ell+1/2}^\eps(p_{i,\ell+1}^\eps-p_{i,\ell}^\eps)(\log u_{i,\ell+1}^\eps
	- \log u_{i,\ell}^\eps) \\
	&\le -\rho c_0\frac{\Delta t}{\Delta x}\sum_{i=1}^n\sum_{\ell\in G}\pi_i
	(p_{i,\ell+1}^\eps-p_{i,\ell}^\eps)(u_{i,\ell+1}^\eps-u_{i,\ell}^\eps).
\end{align*}
Then, inserting definition \eqref{1.p} of $p_{i,\ell}^\eps$
and using the discrete analog \eqref{2.B} of $\pa_x B^{ij}*u_j=B^{ij}*\pa_x u_j$,
\begin{align*}
  & I_4 \le -\rho c_0\frac{\Delta t}{\Delta x}(I_{41}+I_{42}), \quad\mbox{where} \\
	& I_{41} = \sum_{i=1}^n\sum_{\ell\in G}\pi_i a_{ii}(u_{i,\ell+1}^\eps-u_{i,\ell}^\eps)^2, \\
	& I_{42} = \sum_{\substack{i,j=1\\ j\neq i}}^n\sum_{\ell,\ell'\in G}
	\Delta x\pi_i a_{ij}B_{\ell-\ell'}^{ij}(u_{j,\ell'+1}^\eps-u_{j,\ell'}^\eps)
	(u_{i,\ell+1}^\eps-u_{i,\ell}^\eps).
\end{align*}
We insert $(n-1)^{-1}\sum_{j\neq i}1=1$ and $\sum_{\ell'\in G}\Delta x=1$ 
(note that $\m(\torus)=1$) in $I_{41}$
and split the resulting sum in two parts:
\begin{align*}
  I_{41} = \frac{1}{n-1}\sum_{\substack{i,j=1\\ i<j}}^n\sum_{\ell,\ell'\in G}
	\Delta x\pi_i a_{ii}(u_{i,\ell+1}^\eps-u_{i,\ell}^\eps)^2
	+ \frac{1}{n-1}\sum_{\substack{i,j=1\\ i>j}}^n\sum_{\ell,\ell'\in G}
	\Delta x\pi_i a_{ii}(u_{i,\ell+1}^\eps-u_{i,\ell}^\eps)^2.
\end{align*}
We exchange $i$ and $j$ as well as $\ell$ and $\ell'$ in the second term, which leads to
$$
  I_{41} = \frac{1}{n-1}\sum_{\substack{i,j=1\\ i<j}}^n\sum_{\ell,\ell'\in G}
	\Delta x\big[\pi_i a_{ii}(u_{i,\ell+1}^\eps-u_{i,\ell}^\eps)^2
	+ \pi_j a_{jj}(u_{j,\ell'+1}^\eps-u_{j,\ell'}^\eps)^2\big].
$$
Similarly, we distinguish between $i<j$ and $i>j$ in $I_{42}$ and exchange $i$ and $j$
as well as $\ell$ and $\ell'$ in the sum over $i>j$, leading to
\begin{align*}
  I_{42} &= \sum_{\substack{i,j=1\\ i<j}}^n\sum_{\ell,\ell'\in G}
	\Delta x\pi_i a_{ij}B_{\ell-\ell'}^{ij}(u_{j,\ell'+1}^\eps-u_{j,\ell'}^\eps)
	(u_{i,\ell+1}^\eps-u_{i,\ell}^\eps) \\
	&\phantom{xx}{}+ \sum_{\substack{i,j=1\\ i<j}}^n\sum_{\ell,\ell'\in G}
	\Delta x\pi_j a_{ji}B_{\ell'-\ell}^{ji}(u_{i,\ell+1}^\eps-u_{i,\ell}^\eps)
	(u_{j,\ell'+1}^\eps-u_{j,\ell'}^\eps).
\end{align*}
By Remark \ref{rem.symm}, we have $B_{\ell'-\ell}^{ji}=B_{\ell-\ell'}^{ij}$. Therefore,
$$
  I_{42} = \sum_{\substack{i,j=1\\ i<j}}^n\sum_{\ell,\ell'\in G}
	\Delta x(\pi_i a_{ij}+\pi_j a_{ji})B_{\ell-\ell'}^{ij}
	(u_{j,\ell'+1}^\eps-u_{j,\ell'}^\eps)(u_{i,\ell+1}^\eps-u_{i,\ell}^\eps).
$$
The sum of $I_{41}$ and $I_{42}$ can be written as a quadratic form in $\mathrm{D}_\ell u_i^\eps$
and $\mathrm{D}_{\ell'}u_j^\eps$ with the matrix $M^{ij}_{\ell-\ell'}$, defined in
\eqref{2.Mijdis}. This shows that
\begin{align*}
  I_4 \le -\frac{\rho c_0\Delta t}{(n-1)} \sum_{\substack{i,j=1\\i<j}}^n
	\sum_{\ell,\ell'\in G}(\Delta x)^2
  \begin{pmatrix} \mathrm{D}_\ell u_i^\eps \\ \mathrm{D}_{\ell'}u_j^\eps
  \end{pmatrix}^\top
  M^{ij}_{\ell-\ell'}
  \begin{pmatrix} \mathrm{D}_\ell u_i^\eps \\ \mathrm{D}_{\ell'}u_j^\eps
  \end{pmatrix}\le 0.
\end{align*}

Collecting the estimates for $I_3$ and $I_4$ in \eqref{3.I34}, we deduce from
\eqref{3.aux2} the following regularized discrete entropy inequality:
\begin{align}\label{3.rei}
  \rho\mathcal{H}_B(u^\eps) &+ \eps\Delta t\sum_{i=1}^n\|w_i^\eps\|_{1,2,\T}^2
	+ 4\rho\sigma\Delta t\sum_{i=1}^n\pi_i|(u_i^\eps)^{1/2}|_{1,2,\T}^2 \\
	&{}+ \frac{\rho c_0\Delta t}{(n-1)} \sum_{\substack{i,j=1\\i<j}}^n
	\sum_{\ell,\ell'\in G}(\Delta x)^2
  \begin{pmatrix} \mathrm{D}_\ell u_i^\eps \\ \mathrm{D}_{\ell'}u_j^\eps
  \end{pmatrix}^\top M^{ij}_{\ell-\ell'}
  \begin{pmatrix} \mathrm{D}_\ell u_i^\eps \\ \mathrm{D}_{\ell'}u_j^\eps
  \end{pmatrix}
	\le \rho\mathcal{H}_B(u^{k-1}). \nonumber
\end{align}

We proceed with the topological degree argument. We set
$R=1+(\mathcal{H}_B(u^{k-1})/(\eps\Delta t))^{1/2}$. Then \eqref{3.rei}
implies that
$$
  \eps\Delta t\sum_{i=1}^n\|w_i^\eps\|_{1,2,\T}^2 \le \rho\mathcal{H}_B(u^{k-1})
	\le \mathcal{H}_B(u^{k-1}) = \eps\Delta t(R-1)^2 < \eps\Delta tR^2
$$
and hence $w^\eps\not\in\pa Z_R$. We infer that $\operatorname{deg}(I-F,Z_R,0)=1$
and consequently, $F$ admits a fixed point.
Note that we did not use the estimate for $u_i^\eps$ in the
seminorm $|\cdot|_{1,2,\T}$ at this point, such that $\sigma=0$ is admissible here
(and also in the following two subsections).

\subsection{Limit $\eps\to 0$}

There exists a constant $C>0$ such that $C(s-1)\le h(s)$ for all $s\ge 0$.
Hence,
$$
  C\pi_i \Delta x(u_{i,\ell}^\eps-1) \le \pi_i \Delta x h(u_{i,\ell}^\eps)
	\le \mathcal{H}_B(u^\eps) \le \mathcal{H}_B(u^{k-1})
$$
for all $\ell\in G$, $i=1,\ldots,n$. Thus, $(u_{i,\ell}^\eps)$ is bounded in $\eps$
and the Bolzano--Weierstra{\ss} theorem implies the existence of a subsequence
(not relabeled) such that $u_{i,\ell}^\eps\to u_{i,\ell}^k\ge 0$ as $\eps\to 0$.
It follows from \eqref{3.rei} that $\eps w_{i,\ell}^\eps\to 0$. Thus, the
limit $\eps\to 0$ in \eqref{3.weps} shows that $u^k$ is a solution to the 
numerical scheme \eqref{2.fv1}--\eqref{2.fv3}.
Moreover, the limit $\eps\to 0$ in \eqref{3.rei}
leads to the discrete entropy inequality \eqref{2.Bei}.

\subsection{Discrete Rao entropy inequality}

We prove inequality \eqref{2.Rei}. To this end, we multiply
\eqref{2.fv1} by $\Delta t\pi_i p_{i,\ell}^k$ and sum over $\ell\in G$, $i=1,\ldots,n$:
\begin{equation}\label{3.aux4}
  \sum_{i=1}^n\sum_{\ell\in G}\Delta x\pi_i(u_{i,\ell}^k-u_{i,\ell}^{k-1})p_{i,\ell}^k
	+ \sum_{i=1}^n\sum_{\ell\in G}\Delta t\pi_i(\mathcal{F}_{i,\ell+1/2}^k
	- \mathcal{F}_{i,\ell-1/2}^k)p_{i,\ell}^k = 0.
\end{equation}
For the first term in \eqref{3.aux4}, we use the definition of $p_{i,\ell}^k$:
\begin{align*}
  & \sum_{i=1}^n\sum_{\ell\in G}\Delta x\pi_i(u_{i,\ell}^k-u_{i,\ell}^{k-1})p_{i,\ell}^k 
	= I_5+I_6, \quad\mbox{where} \\
  & I_5 = \sum_{i=1}^n\sum_{\ell\in G}\Delta x \pi_i a_{ii}(u_{i,\ell}^k-u_{i,\ell}^{k-1})
	u_{i,\ell}^k, \\
  & I_6 = \sum_{\substack{i,j=1\\j\neq i}}^n\sum_{\ell,\ell'\in G}(\Delta x)^2 
	\pi_i a_{ij} B^{ij}_{\ell-\ell'}
	(u_{i,\ell}^k-u_{i,\ell}^{k-1})u_{j,\ell'}^{k}.
\end{align*}
We rewrite $I_5$ and $I_6$ according to
\begin{align*}
  I_5  &= \frac12 \sum_{i=1}^n\sum_{\ell\in G}\Delta x \pi_i a_{ii} 
	\big((u_{i,\ell}^k)^2-(u_{i,\ell}^{k-1})^2\big) + \frac12 \sum_{i=1}^n\sum_{\ell\in G}\Delta x \pi_i a_{ii} \big(u^k_{i,\ell}-u^{k-1}_{i,\ell}\big)^2, \\
   I_6 	&= \frac12\sum_{\substack{i,j=1\\j\neq i}}^n\sum_{\ell,\ell'\in G}
	(\Delta x)^2\pi_i a_{ij}
	B^{ij}_{\ell-\ell'}(u_{i,\ell}^ku_{j,\ell'}^k - u_{i,\ell}^{k-1}u_{j,\ell'}^{k-1}) \\
	&\phantom{xx}{}+ \frac12\sum_{\substack{i,j=1\\j\neq i}}^n\sum_{\ell,\ell'\in G}
	(\Delta x)^2\pi_i a_{ij}
	B^{ij}_{\ell-\ell'}(u_{i,\ell}^k-u_{i,\ell}^{k-1})(u_{j,\ell'}^k-u_{j,\ell'}^{k-1}).
\end{align*}
Combining the second terms in $I_5$ and $I_6$, using similar computations as for $I_4$ in Section \ref{sec.fixed}, and applying Hypothesis (H3) show that the second term of $I_5+I_6$ is nonnegative so that
\begin{align*}
	I_5+I_6 &\ge \frac12 \sum_{i=1}^n\sum_{\ell\in G}\Delta x \pi_i a_{ii} 
	\big((u_{i,\ell}^k)^2-(u_{i,\ell}^{k-1})^2\big)\\
	&+ \frac12\sum_{\substack{i,j=1\\j\neq i}}^n\sum_{\ell,\ell'\in G}
	(\Delta x)^2\pi_i a_{ij}
	B^{ij}_{\ell-\ell'}(u_{i,\ell}^ku_{j,\ell'}^k - u_{i,\ell}^{k-1}u_{j,\ell'}^{k-1}).
\end{align*}
Then it holds that
\begin{align*}
  \sum_{i=1}^n\sum_{\ell\in G}\Delta x\pi_i(u_{i,\ell}^k-u_{i,\ell}^{k-1})p_{i,\ell}^k 
	\geq \mathcal{H}_R(u^k)-\mathcal{H}_R(u^{k-1}).
\end{align*}

Now, we split the second term in \eqref{3.aux4} again into two parts:
\begin{align*}
  & \sum_{i=1}^n\sum_{\ell\in G}\Delta t\pi_i(\mathcal{F}_{i,\ell+1/2}^k
	- \mathcal{F}_{i,\ell-1/2}^k)p_{i,\ell}^k = I_7 + I_8,\quad\mbox{where} \\
	& I_7 = -\sigma\Delta t\sum_{i=1}^n\sum_{\ell\in G}\pi_i
	\bigg(\frac{u_{i,\ell+1}^k-u_{i,\ell}^k}{\Delta x} 
	- \frac{u_{i,\ell}^k-u_{i,\ell-1}^k}{\Delta x}\bigg)p_{i,\ell}^k, \\
	& I_8 = -\Delta t\sum_{i=1}^n\sum_{\ell\in G}\pi_i
	\bigg(u_{i,\ell+1/2}^k\frac{p_{i,\ell+1}^k-p_{i,\ell}^k}{\Delta x}
	- u_{i,\ell-1/2}^k\frac{p_{i,\ell}^k-p_{i,\ell-1}^k}{\Delta x}\bigg)p_{i,\ell}^k.
\end{align*}
We reformulate $I_7$ by using a discrete integration by parts:
\begin{align*}
  I_7 = \sigma\Delta t\sum_{i=1}^n\sum_{\ell\in G}\pi_i
	\frac{u_{i,\ell+1}^k-u_{i,\ell}^k}{\Delta x}(p_{i,\ell+1}^k-p_{i,\ell}^k).
\end{align*}
Then, with similar computations as for $I_4$ in Section \ref{sec.fixed}, we obtain
\begin{align*}
  I_7 = \frac{\sigma\Delta t}{(n-1)}\sum_{\substack{i,j=1\\i<j}}^n
	\sum_{\ell,\ell'\in G}(\Delta x)^2
  \begin{pmatrix} \mathrm{D}_\ell u_i^{k} \\ \mathrm{D}_{\ell'}u_j^{k}
  \end{pmatrix}^\top M^{ij}_{\ell-\ell'}
  \begin{pmatrix} \mathrm{D}_\ell u_i^{k} \\ \mathrm{D}_{\ell'}u_j^{k}
  \end{pmatrix}\ge 0.
\end{align*}
Finally, the term $I_8$ can be rewritten as
\begin{align*}
  I_8 &= \Delta t\sum_{i=1}^n\sum_{\ell\in G}\pi_i u_{i,\ell+1/2}^k
	\frac{p_{i,\ell+1}^k-p_{i,\ell}^k}{\Delta x}(p_{i,\ell+1}^k-p_{i,\ell}^k) 
	= \Delta t\sum_{i=1}^n\sum_{\ell\in G}\pi_i\Delta x
	\big|(u_{i,\ell+1/2}^k)^{1/2}\mathrm{D}_\ell p_{i}^k\big|^2.
\end{align*}
Hence, we infer from \eqref{3.aux4} that
\begin{align*}
  \mathcal{H}_R(u^k) &+ \Delta t\sum_{i=1}^n\sum_{\ell\in G}\pi_i\Delta x
	\big|(u_{i,\ell+1/2}^k)^{1/2}\mathrm{D}_\ell p_{i}^k\big|^2 \\
	&{}+ \frac{\sigma\Delta t}{(n-1)}\sum_{\substack{i,j=1\\i<j}}^n
	\sum_{\ell,\ell'\in G}(\Delta x)^2
  \begin{pmatrix} \mathrm{D}_\ell u_i^{k} \\ \mathrm{D}_{\ell'}u_j^{k}
  \end{pmatrix}^\top M^{ij}_{\ell-\ell'}
  \begin{pmatrix} \mathrm{D}_\ell u_i^{k} \\ \mathrm{D}_{\ell'}u_j^{k}
  \end{pmatrix}
	\le \mathcal{H}_R(u^{k-1}),
\end{align*}
which proves \eqref{2.Rei}. 

Finally, conservation of mass follows from summing \eqref{2.fv1} over $\ell\in G$
and observing that the sum over the numerical fluxes vanishes.
This ends the proof of Theorem \ref{thm.ex}.


\section{Proof of Theorem \ref{thm.conv}}\label{sec.conv}

To prove the convergence of the scheme, we derive first some uniform estimates
and then apply a discrete Aubin--Lions compactness lemma.

\subsection{Uniform estimates}

Let $(u_m)_{m\in\N}$ be a sequence of finite-volume solutions to 
\eqref{2.fv1}--\eqref{2.fv3} associated to the mesh $\mathcal{D}_m$
and constructed in Theorem \ref{thm.ex}.
The conservation of mass and the discrete entropy inequalities \eqref{2.Bei} 
and \eqref{2.Rei} show that, after summing over $k=1,\ldots,N_T^m$,
\begin{equation}\label{4.est0}
  \max_{k=1,\ldots,N_T^m}\|u^k_i\|_{0,2,\T_m}^2
	+ \sum_{k=1}^{N_T^m}\Delta t_m\|(u_i^k)^{1/2}\|_{1,2,\T_m}^2 \le C, \quad i=1,\ldots,n,
\end{equation}
where $C>0$ denotes here and in the following a constant independent of the
mesh size $\eta_m=\max\{\Delta x_m,$ $\Delta t_m\}$,
but possibly depending on $u^0$ and $T$. Because of the positive definiteness
of $M^{ij}_{\ell-\ell'}$, we conclude a bound for $u_i^k$ in the norm $\|\cdot\|_{1,2,\T_m}$.

\begin{lemma}
Let the assumptions of Theorem \ref{thm.conv} hold. Then there exists 
$C>0$ independent of $\eta_m$ (but depending on the positive definite constant $c_M$)
such that for all $m\in\N$, $i=1,\ldots,n$,
\begin{equation}\label{4.H1}
  \sum_{k=1}^{N_T^m}\Delta t_m\|u_i^k\|_{1,2,\T_m}^2\le C.
\end{equation}
\end{lemma}

\begin{proof}
We infer from \eqref{2.Bei} that
$$
  \frac{c_0}{n-1}\sum_{k=1}^{N_T^m}\Delta t_m\sum_{\substack{i,j=1\\i<j}}^n
	\sum_{\ell,\ell'\in G_m}(\Delta x)^2
  \begin{pmatrix} \mathrm{D}_\ell u_i^{k} \\ \mathrm{D}_{\ell'}u_j^{k}
  \end{pmatrix}^\top M^{ij}_{\ell-\ell'}
  \begin{pmatrix} \mathrm{D}_\ell u_i^{k} \\ \mathrm{D}_{\ell'}u_j^{k}
  \end{pmatrix} \le \mathcal{H}_{B}(u^0),
$$
Since $M^{ij}_{\ell-\ell'}$ is positive definite with constant $c_M>0$,
\begin{align*}
  \frac{c_0}{n-1}&\sum_{\substack{i,j=1\\i<j}}^n\sum_{\ell,\ell'\in G_m}(\Delta x)^2
  \begin{pmatrix} \mathrm{D}_\ell u_i^{k} \\ \mathrm{D}_{\ell'}u_j^{k}
  \end{pmatrix}^\top M^{ij}_{\ell-\ell'}
  \begin{pmatrix} \mathrm{D}_\ell u_i^{k} \\ \mathrm{D}_{\ell'}u_j^{k}
  \end{pmatrix} \\
	&\ge \frac{c_M c_0}{n-1}\sum_{\substack{i,j=1\\i<j}}^n\sum_{\ell,\ell'\in G_m}(\Delta x)^2
	\big(|\mathrm{D}_\ell u_i^{k}|^2 + |\mathrm{D}_{\ell'}u_j^{k}|^2\big) \\
	&= c_M c_0\sum_{i=1}^n\sum_{\ell\in G_m}\Delta x|\mathrm{D}_\ell u_i^{k}|^2
	+ c_M c_0\sum_{j=1}^n\sum_{\ell'\in G_m}\Delta x|\mathrm{D}_{\ell'}u_j^{k}|^2 \\
	&= 2c_M c_0\sum_{i=1}^n\sum_{\ell\in G_m}\Delta x|\mathrm{D}_\ell u_i^{k}|^2.
\end{align*}
Together with the first bound in \eqref{4.est0}, this finishes the proof.
\end{proof}

\begin{lemma}\label{lem.est}
Let the assumptions of Theorem \ref{thm.conv} hold. Then there exists a constant $C>0$ independent of $\eta_m$ (but depending on $\sigma$)
such that for all $m\in\N$, $i=1,\ldots,n$,
\begin{equation*}
	\sum_{k=1}^{N_T^m}\Delta t_m\|u_i^k\|_{1,1,\T_m}^2
	+ \sum_{k=1}^{N_T^m}\Delta t_m\|u_i^k\|_{0,\infty,\T_m}^2 \le C.
\end{equation*}
Moreover, there exists another constant, still denoted by $C>0$ and independent of $\eta_m$, such that
\begin{align}\label{4.pH1}
  \sum_{k=1}^{N_T^m} \Delta t_m|p_i^k|_{1,2,\T_m}^2 \leq C.
\end{align}
\end{lemma}

\begin{proof}
As $\m(\torus)=1$, thanks to the Cauchy--Schwarz inequality,
\begin{align*}
  |u_{i}^k|_{1,1,\T_m} &= \sum_{\ell\in G_m}|u_{i,\ell+1}^k-u_{i,\ell}^k|
	\leq |u_i^k|_{1,2,\T_m}.
\end{align*}
Using \eqref{4.H1}, this shows that
\begin{align*}
  \sum_{k=1}^{N^m_T}&\Delta t_m\|u_i^k\|_{1,1,\T_m}^2
	\le 2\sum_{k=1}^{N^m_T}\Delta t_m\big(\|u_i^k\|_{0,1,\T_m}^2 + |u_i^k|_{1,1,\T_m}^2\big) \\
	&\le 2T\max_{k=1,\ldots,N^m_T}\|u_i^k\|_{0,1,\T_m}^2
	+ 2\sum_{k=1}^{N^m_T}\Delta t_m|u_i^k|_{1,2,\T_m}^2 \le C(u^0,T).
\end{align*}
To show the discrete $L^\infty(\torus)$ bound, we apply the continuity of 
the embedding $\operatorname{BV}(\torus)\hookrightarrow L^\infty(\torus)$
(in one space dimension). We conclude that, for $i=1,\ldots,n$,
$$
  \sum_{k=1}^{N^m_T}\Delta t_m\|u_i^k\|_{0,\infty,\T_m}^2
	\le C\sum_{k=1}^{N^m_T}\Delta t_m\|u_i^k\|_{{\rm BV}(\torus)}^2
	= C\sum_{k=1}^{N^m_T}\Delta t_m\|u_i^k\|_{1,1,\T_m}^2 \le C(u^0,T).
$$
For the last part, we estimate
\begin{align*}
  |p_i^k|_{1,2,\T_m}^2
	&= \sum_{\ell\in G_m}\Delta x_m\bigg|
	\frac{p_{i,\ell+1}^k-p_{i,\ell}^k}{\Delta x_m}\bigg|^2 \\
	&\leq C a_{ii}^2 |u^k_i|^2_{1,2,\T_m} 
	+ C\sum_{\ell\in G_m}\Delta x_m\bigg|\sum_{\substack{j=1\\j\neq i}}^n\sum_{\ell'\in G_m}
	\Delta x_m a_{ij}
	\frac{B^{ij}_{\ell+1-\ell'}-B^{ij}_{\ell-\ell'}}{\Delta x_m}\, u_{j,\ell'}^k\bigg|^2 \\
	&\leq C |u^k_i|^2_{1,2,\T_m} + C\sum_{\ell\in G_m}\Delta x_m
	\bigg|\sum_{\substack{j=1\\j\neq i}}^n\sum_{\ell'\in G_m}\Delta x_m a_{ij}
	B^{ij}_{\ell-\ell'}\frac{u_{j,\ell'+1}^k-u_{j,\ell'}^k}{\Delta x_m}\bigg|^2 \\
	&\leq C|u^k_i|^2_{1,2,\T_m} 
	+ C\sum_{\ell\in G_m}\Delta x_m\bigg|\sum_{\substack{j=1\\j\neq i}}^n\sum_{\ell'\in G_m}
	\Delta x_m a_{ij}B^{ij}_{\ell-\ell'}\mathrm{D}_{\ell'} u_j^k\bigg|^2.
\end{align*}
Then we deduce from the elementary inequality $(\sum_{\substack{j=1,\,j \neq i}}^n a_j)^2
\le (n-1)\sum_{\substack{j=1,\,j\neq i}}^n a_j^2$
for $a_j\in\R$ and the discrete Young convolution inequality in Lemma \ref{lem.young} that
\begin{align*}
  &\sum_{\ell\in G_m}\Delta x_m\bigg|\sum_{\substack{j=1\\j\neq i}}^n\sum_{\ell'\in G_m}
	\Delta x_m a_{ij}B^{ij}_{\ell-\ell'}\mathrm{D}_{\ell'} u_j^k\bigg|^2 \\
	&\phantom{x}\le (n-1)\sum_{\substack{j=1\\j\neq i}}^n\sum_{\ell\in G_m}\Delta x_m
	\bigg(\sum_{\ell'\in G_m}\Delta x_m a_{ij} B^{ij}_{\ell-\ell'}\mathrm{D}_{\ell'} u_j^k
	\bigg)^2
	\le C\sum_{\substack{j=1\\j\neq i}}^n \|B^{ij}\|_{L^2(\torus)}^2|u_j^k|_{1,1,\T_m}^2.
\end{align*}
Summing over $k$, we infer that
\begin{align*}
  \sum_{k=1}^{N_T^m} \Delta t_m|p_i^k|_{1,2,\T_m}^2
	&\leq C\bigg\{\sum_{i=1}^n \sum_{k=1}^{N_T^m} \Delta t_m |u^k_i|^2_{1,2,\T_m} 
	+ \sum_{\substack{j=1\\j\neq i}}^n\bigg(\|B^{ij}\|_{L^2(\torus)}^2
	\sum_{k=1}^{N_T^m}\Delta t_m|u_j^k|_{1,1,\T_m}^2\bigg)\bigg\} \le C,
\end{align*}
where we used Lemma \ref{lem.est} for the last inequality. At this point, we need the discrete $L^2(0,T;H^1(\torus))$ bound of $(u_{m,i})$. This ends the proof.
\end{proof}

Next, we show a uniform bound for the discrete time derivative.

\begin{lemma}\label{lem.time}
Let the assumptions of Theorem \ref{thm.conv} hold. Then there exists 
$C>0$ independent of $\eta_m$ such that for all $m\in\N$, $i=1,\ldots,n$,
$$
  \sum_{k=1}^{N_T^m}\Delta t_m\bigg\|\frac{u_i^k-u_i^{k-1}}{\Delta t_m}
	\bigg\|_{-1,2,\T_m}^{4/3} \le C.
$$
\end{lemma}

\begin{proof}
Let $\phi=(\phi_\ell)_{\ell\in G_m}\in\mathcal{V}_{\T_m}$ be such that $\|\phi\|_{1,2,\T_m}=1$. We multiply \eqref{2.fv1} by $\phi_\ell$, sum over
$\ell\in G_m$, and use discrete integration by parts:
\begin{align}\label{4.I78}
  \sum_{\ell\in G_m}&\Delta x_m\frac{u_{i,\ell}^k-u_{i,\ell}^{k-1}}{\Delta t_m}\phi_\ell
	= \sigma\sum_{\ell\in G_m}\bigg(\frac{u_{i,\ell+1}^k-u_{i,\ell}^k}{\Delta x_m}
	- \frac{u_{i,\ell}^k-u_{i,\ell-1}^k}{\Delta x_m}\bigg)\phi_\ell \\
	&\phantom{xx}{}+ \sum_{\ell\in G_m}\bigg(
	u_{i,\ell+1/2}^k\frac{p_{i,\ell+1}^k-p_{i,\ell}^k}{\Delta x_m}
	- u_{i,\ell-1/2}^k\frac{p_{i,\ell}^k-p_{i,\ell-1}^k}{\Delta x_m}\bigg)\phi_\ell 
	\nonumber \\
	&= -\sigma\sum_{\ell\in G_m}\Delta x_m\frac{u_{i,\ell+1}^k-u_{i,\ell}^k}{\Delta x_m}
	\frac{\phi_{\ell+1}-\phi_\ell}{\Delta x_m}
	- \sum_{\ell\in G_m}\Delta x_m u_{i,\ell+1/2}^k
	\frac{p_{i,\ell+1}^k-p_{i,\ell}^k}{\Delta x_m}\frac{\phi_{\ell+1}-\phi_\ell}{\Delta x_m} 
	\nonumber \\
	&=: I_9 + I_{10}. \nonumber
\end{align}
By the Cauchy--Schwarz inequality,
\begin{align*}
  |I_9| &\le \sigma\sum_{\ell\in G_m} \Delta x_m\big((u_{i,\ell+1}^k)^{1/2}
	+ (u_{i,\ell}^k)^{1/2}\big)\bigg|
	\frac{(u_{i,\ell+1}^k)^{1/2}-(u_{i,\ell}^k)^{1/2}}{\Delta x_m}\bigg|
	\bigg|\frac{\phi_{\ell+1}-\phi_\ell}{\Delta x_m}\bigg| \\
	&\le 2\sigma\|(u_i^k)^{1/2}\|_{0,\infty,\T_m}|(u_i^k)^{1/2}|_{1,2,\T_m}
	|\phi|_{1,2,\T_m}.
\end{align*}
Furthermore, using $(u_{i,\ell+1/2}^k)^{1/2}
\le\max\{(u_{i,\ell}^k)^{1/2},(u_{i,\ell+1}^k)^{1/2}\} 
\le \|(u^k_i)^{1/2}\|_{0,\infty,\T_m}$,
\begin{align*}
  |I_{10}| &\le \sum_{\ell\in G_m}\Delta x_m \big|(u_{i,\ell+1/2}^k)^{1/2}\big|
	\bigg|(u_{i,\ell+1/2}^k)^{1/2}\frac{p_{i,\ell+1}^k-p_{i,\ell}^k}{\Delta x_m}\bigg|
	\bigg|\frac{\phi_{\ell+1}-\phi_\ell}{\Delta x_m}\bigg| \\
	&\le \|(u_i^k)^{1/2}\|_{0,\infty,\T_m}\bigg(\sum_{\ell\in G_m}\Delta x_m
	\bigg|(u_{i,\ell+1/2}^k)^{1/2}\frac{p_{i,\ell+1}^k-p_{i,\ell}^k}{\Delta x_m}\bigg|^2
	\bigg)^{1/2}|\phi|_{1,2,\T_m}.
\end{align*}
Applying the elementary inequality $(a+b)^r \leq C(a^r+b^r)$ for all $a, b \geq 0$ 
and $r>1$, inserting the previous estimates into \eqref{4.I78}, 
and using H\"older's inequality, we find that
\begin{align*}
  \sum_{k=1}^{N^m_T}& \Delta t_m\bigg\|\frac{u_{i}^k-u_{i}^{k-1}}{\Delta t_m}
	\bigg\|_{-1,2,\T_m}^{4/3}
	= \sum_{k=1}^{N^m_T}\Delta t_m \sup_{\|\phi\|_{1,2,\T_m}=1}
	\bigg|\sum_{\ell\in G_m} \Delta x_m \frac{u_{i,\ell}^k-u_{i,\ell}^{k-1}}{\Delta t_m}
	\phi_\ell\bigg|^{4/3} \\
	&\le C\sum_{k=1}^{N^m_T} \Delta t_m\|(u_i^k)^{1/2}\|_{0,\infty,\T_m}^{4/3}
	|(u_i^k)^{1/2}|_{1,2,\T_m}^{4/3} \\
	&\phantom{xx}{}+ C\sum_{k=1}^{N^m_T} \Delta t_m\|(u_i^k)^{1/2}\|_{0,\infty,\T_m}^{4/3}
	\bigg(\sum_{\ell\in G_m}\Delta x_m
	\bigg|(u_{i,\ell+1/2}^k)^{1/2}\frac{p_{i,\ell+1}^k-p_{i,\ell}^k}{\Delta x_m}\bigg|^2
	\bigg)^{2/3} \\
	&\le C\bigg(\sum_{k=1}^{N^m_T}\Delta t_m\|(u_i^k)^{1/2}\|_{0,\infty,\T_m}^4\bigg)^{1/3}
	\bigg(\sum_{k=1}^{N^m_T} \Delta t_m|(u_i^k)^{1/2}|_{1,2,\T_m}^2\bigg)^{2/3} \\
	&\phantom{xx}{}+ C\bigg(\sum_{k=1}^{N^m_T}\Delta t_m
	\|(u_i^k)^{1/2}\|_{0,\infty,\T_m}^4\bigg)^{1/3}
	\bigg(\sum_{k=1}^{N^m_T}\Delta t_m \sum_{\ell\in G_m}\Delta x_m
	\bigg|(u_{i,\ell+1/2}^k)^{1/2}\frac{p_{i,\ell+1}^k-p_{i,\ell}^k}{\Delta x_m}\bigg|^2
	\bigg)^{2/3} \\
	&\le C(u^0,T),
\end{align*}
and the last bound follows from Lemma \ref{lem.est} and the discrete Rao entropy 
inequality \eqref{2.Rei}.
\end{proof}

\subsection{Compactness}

We claim now that the estimates
from Lemmas \ref{lem.est} and \ref{lem.time} are sufficient to conclude the
relative compactness of $(u_m)_{m \in \N}$. In fact,
the result follows from the discrete Aubin--Lions lemma \cite[Theorem 3.4]{GaLa12}
if the following two properties are satisfied: 
\begin{itemize}
\item For any 
$(v_m)_{m \in \N}\subset\mathcal{V}_{\T_m}$ such that $\sup_{m\in\N}\|v_m\|_{1,2,\T_m}\le C$ for
some $C>0$, there exists $v\in L^2(\torus)$ satisfing, up to a subsequence,
$v_m\to v$ strongly in $L^2(\torus)$. This property follows from
\cite[Theorem 14.1]{EGH00}. 
\item If $v_m\to v$ strongly in $L^2(\torus)$
and $\|v_m\|_{-1,2,\T_m}\to 0$ as $m\to\infty$, then $v=0$. This property
can be replaced by the condition that $\|\cdot\|_{1,2,\T_m}$ and 
$\|\cdot\|_{-1,2,\T_m}$ are dual norms with respect
to the $L^2(\torus)$ norm, which is the case \cite[Remark 6]{GaLa12}.
A more detailed proof can be found in \cite[Prop.~10]{JuZu21}.
\end{itemize}
Hence, it follows from \cite[Theorem 3.4]{GaLa12} that
there exists a subsequence, which is not relabeled, such that
\begin{equation*}
  u_{m,i}\to u_i \quad\mbox{strongly in }L^1(0,T;L^2(\torus))\mbox{ as }m\to\infty.
\end{equation*}
Let us now adapt in our case the Gagliardo--Nirenberg inequality. Let $k=1,\ldots,N_T^m$ be fixed. We first apply Lemma~\ref{lem.aux3} with $s=p=2$:
\begin{align*}
  \|u^k_{m,i}\|_{0,\infty,\T_m} \leq C \|u^k_{m,i}\|^{1/2}_{1,2,\T_m} \|u^k_{m,i}\|^{1/2}_{0,2,\T}.
\end{align*}
Then it follows from the H\"older inequality
\begin{align*}
  \|u^k_{m,i}\|_{0,6,\T_m} \leq \|u^k_{m,i}\|^{2/3}_{0,\infty,\T_m} \|(u_{m,i}^k)^{1/3} \|_{0,6,\T_m} = \|u^k_{m,i}\|^{2/3}_{0,\infty,\T_m} \|u_{m,i}^k \|_{0,2,\T_m}^{1/3}
\end{align*}
that
\begin{align*}
  \|u^k_{m,i}\|_{0,6,\T_m} \leq C \|u^k_{m,i}\|^{1/3}_{1,2,\T_m} \, \|u^k_{m,i}\|^{2/3}_{0,2,\T}.
\end{align*}
Therefore,
\begin{align*}
  \sum_{k=1}^{N_T} \Delta t_m \|u^k_{m,i}\|^6_{0,6,\T_m} \leq C \, \max_{k=1,\ldots,N_T} \|u_{m,i}\|^4_{0,2,\T_m} \, \sum_{k=1}^{N_T} \Delta t_m \, \|u^k_{m,i}\|^2_{1,2,\T_m}.
\end{align*}
Recalling estimates \eqref{4.est0} and \eqref{4.H1}, we conclude that $(u_{m,i})_{m \in \N}$ is uniformly bounded in $L^6(\torus)$. The convergence dominated theorem implies that, up to a subsequence, for every $p<6$,
\begin{equation*}
 u_{m,i}\to u_i \quad\mbox{strongly in }L^p(Q_T)\mbox{ as }m\to\infty.
\end{equation*}
Lemma \ref{lem.est} implies that the sequence of discrete derivatives
$(\pa_x^m u_{m,i})_{m \in \N}$ is bounded in $L^{2}(Q_T)$. Thus, there
exists a subsequence (not relabeled) such that $\pa_x^m u_{m,i}\rightharpoonup v_i$
weakly in $L^{2}(Q_T)$, and the proof of \cite[Lemma 4.4]{CLP03}
allows us to identify $v_i=\pa_x u_i$.

\begin{lemma}\label{lem.convp}
The following convergences hold, up to subsequences, as $m\to\infty$,
\begin{align*}
  p_{m,i} \to p_i(u) &\quad\mbox{strongly in }L^2(Q_T), \\
	\pa_x p_{m,i} \rightharpoonup \pa_x p_i(u) &\quad\mbox{weakly in }L^2(Q_T), \quad
	i=1,\ldots,n.
\end{align*}
\end{lemma}

\begin{proof}
We follow the strategy of \cite[Corollary 14]{HeZu22}. First, we rewrite $p_{i,\ell}^k$ defined in \eqref{2.fv3}. By a change of variables, we have
\begin{align*}
  p_{i,\ell}^k &= a_{ii}u_{m,i,\ell}^k + \sum_{\substack{j=1 \\ j\neq i}}^n\sum_{\ell'\in G_m}
	a_{ij}\bigg(\int_{K_{\ell-\ell'}}B^{ij}(y)dy\bigg) u_{m,j,\ell'}^k \\
	&= a_{ii}u_{m,i,\ell}^k + \sum_{\substack{j=1 \\ j\neq i}}^n\sum_{\ell'\in G_m}
	a_{ij}\int_{K_{\ell'}}B^{ij}(x_\ell-z)u_{m,j}^k(z)dz \\
	&= a_{ii}u_{m,i}^k(x_\ell) + \sum_{\substack{j=1 \\ j\neq i}}^n a_{ij}
	(B^{ij}*u_{m,j}^k)(x_\ell).
\end{align*}
We introduce the piecewise constant function $Q_m^{ij}$ by setting
$Q_m^{ij}:=(B^{ij}*u_{m,j})(x_\ell)$ in $K_\ell$ for $\ell\in G_m$. Then
$$
  p_i(u)-p_{m,i} = a_{ii}(u_i-u_{m,i}) + \sum_{\substack{j=1\\ j\neq i}}^n a_{ij}
	(B^{ij}*u_j - Q_m^{ij}).
$$
Since we know that $u_i-u_{m,i}\to 0$ strongly in $L^2(Q_T)$, it is sufficient
to prove that $B^{ij}*u_j - Q_m^{ij}\to 0$ strongly in $L^2(Q_T)$. For this,
we write 
$$
  (B^{ij}*u_j - Q_m^{ij})(x,t) = B^{ij}*(u_j-u_{m,j})(x,t)
	+ \int_\torus(B^{ij}(x-y)-B^{ij}(x_\ell-y))u_{m,j}(y,t)dy.
$$
By Young's convolution inequality, we have
$$
  \|B^{ij}*(u_j-u_{m,j})\|_{L^2(Q_T)} \le \|B^{ij}\|_{L^1(\torus)}
	\|u_j-u_{m,j}\|_{L^2(Q_T)} \to 0.
$$
Setting $\xi(x,y)=B^{ij}(x-y)-B^{ij}(x_\ell-y)$ for $x\in K_\ell$ and $y\in\torus$,
we estimate
\begin{align*}
  \bigg\|\int_\torus\xi(\cdot,y)u_{m,j}(y,t)dy\bigg\|_{L^2(Q_T)}^2
	&\le \int_\torus\|\xi(x,\cdot)\|_{L^2(\torus)}^2 dx\|u_{m,j}\|_{L^2(Q_T)}^2 \\
	&\le \sup_{|z|\le\Delta x_m}\|B^{ij}(z+\cdot)-B^{ij}\|_{L^2(\torus)}^2
	\|u_{m,j}\|_{L^2(Q_T)}^2.
\end{align*}
Since $(u_{m,j})$ is bounded in $L^2(Q_T)$, it remains to verify that
the first factor converges to zero as $\Delta x_m\to 0$. This follows from the
density of continuous functions in $L^2(\torus)$. Indeed, let $\eps>0$ and
$B_\eps^{ij}$ be continuous such that $\|B_\eps^{ij}-B^{ij}\|_{L^2(\torus)}\le\eps$.
Then
\begin{align*}
   \sup_{|z|\le\Delta x_m}&\|B^{ij}(z+\cdot)-B^{ij}\|_{L^2(\torus)}
	\le \sup_{|z|\le\Delta x_m}\|B^{ij}(z+\cdot)-B_\eps^{ij}(z+\cdot)\|_{L^2(\torus)} \\
	&\phantom{xx}{}+ \sup_{|z|\le\Delta x_m}\|B^{ij}_\eps(z+\cdot)-B^{ij}_\eps\|_{L^2(\torus)}
	+ \|B^{ij}_\eps-B^{ij}\|_{L^2(\torus)} \\
	&\le 2\eps + \sup_{|z|\le\Delta x_m}\|B^{ij}_\eps(z+\cdot)-B^{ij}_\eps\|_{L^2(\torus)}.
\end{align*}
The last term is smaller than $\eps$ if we choose $\Delta x_m$ sufficiently small.
We have shown that $\sup_{|z|\le\Delta x_m}\|B^{ij}(z+\cdot)-B^{ij}\|_{L^2(\torus)}^2\to 0$
as $m\to\infty$ and $B^{ij}*u_j - Q_m^{ij}\to 0$ strongly in $L^2(Q_T)$.
This proves the first part of the lemma.

Thanks to \eqref{4.pH1}, we have shown that $(\pa_x^m p_{m,i})_{m \in \N}$ is bounded in $L^2(Q_T)$.
Hence, up to a subsequence, $\pa_x^m p_{m,i}\rightharpoonup z$ weakly in 
$L^2(Q_T)$. The first part of the proof shows that $z=\pa_x p_i(u)$,
finishing the proof.
\end{proof}

\subsection{Convergence of the scheme}

We show that the limit $u=(u_1,\ldots,u_n)$ of the finite-volume solutions
is a weak solution to \eqref{1.eq}--\eqref{1.ic}. 
Let $i\in\{1,\ldots,n\}$ be fixed, let $\psi_i\in C_0^\infty(\torus\times[0,T))$
be given, and let $\eta_m=\max\{\Delta x_m,\Delta t_m\}$. 
We set $\psi_{i,\ell}^k:=\psi_i(x_\ell,t_k)$ and multiply \eqref{2.fv1} by
$\Delta t_m\psi_{i,\ell}^{k-1}$ and sum over $\ell\in G_m$, $k=1,\ldots,N_T^m$.
This yields $F_1^m+F_2^m+F_3^m=0$, where
\begin{align*}
  F_1^m &= \sum_{k=1}^{N_T^m}\sum_{\ell\in G_m}\Delta x_m
	(u_{i,\ell}^k-u_{i,\ell}^{k-1})\psi_{i,\ell}^{k-1}, \\
	F_2^m &= -\sigma\sum_{k=1}^{N_T^m}\Delta t_m\sum_{\ell\in G_m}\bigg(
	\frac{u_{i,\ell+1}^k-u_{i,\ell}^k}{\Delta x_m}
	- \frac{u_{i,\ell}^k-u_{i,\ell-1}^k}{\Delta x_m}\bigg)\psi_{i,\ell}^{k-1}, \\
	F_3^m &= -\sum_{k=1}^{N_T^m}\Delta t_m\sum_{\ell\in G_m}\bigg(
	u_{i,\ell+1/2}^k\frac{p_{i,\ell+1}^k-p_{i,\ell}^k}{\Delta x_m}
	- u_{i,\ell-1/2}^k\frac{p_{i,\ell}^k-p_{i,\ell-1}^k}{\Delta x_m}
	\bigg)\psi_{i,\ell}^{k-1}.
\end{align*}
Furthermore, we introduce the terms
\begin{align*}
  F_{10}^m &= -\int_0^T\int_\torus u_{m,i}\pa_t\psi_i dxdt 
  - \int_\torus u_{m,i}(x,0)\psi_i(x,0)dx, \\
	F_{20}^m &= \sigma \int_0^T\int_\torus\pa_x^m u_{m,i}\pa_x\psi_i dxdt, \\
	F_{30}^m &= \int_0^T\int_\torus u_{m,i}\pa_x^m p_{m,i}\pa_x\psi_i dxdt.
\end{align*}

\begin{lemma}\label{lem.aux1}
Let the assumptions of Theorem \ref{thm.conv} hold. Then it holds that, 
as $m\to\infty$,
\begin{align}
  F_{10}^m &\to -\int_0^T\int_\torus u_i\pa_t\psi_i dxdt
	- \int_\torus u_i^0(x)\psi_i(x,0)dx, \label{4.F10} \\
	F_{20}^m &\to \sigma\int_0^T\int_\torus\pa_x u_i\pa_x\psi_i dxdt, \label{4.F20} \\
	F_{30}^m &\to \int_0^T\int_\torus u_i\pa_x p_i(u)\pa_x \psi_i dxdt. \label{4.F30}
\end{align}
\end{lemma}

\begin{proof}
The strong convergence of $(u_{m,i})_{m\in\N}$ and
the weak convergence of $(\pa_x^m u_{m,i})_{m\in\N}$ in $L^{2}(Q_T)$
as well as the fact that $u_{m,i}(x,0) = (\Delta x_m)^{-1} \int_{K_\ell} u_i^0(z)dz$ for $x \in K_\ell$ and $\ell \in G$ immediately show
convergences \eqref{4.F10} and \eqref{4.F20}. It remains to verify \eqref{4.F30}. We know from Lemma \ref{lem.convp} that $\pa_x^m p_{m,i}\rightharpoonup\pa_x p_i(u)$
weakly in $L^2(Q_T)$.
Since $u_{m,i}\to u_i$ strongly in $L^2(Q_T)$,
this implies that
\begin{equation*}
  u_{m,i}\pa_x^m p_{m,i}\rightharpoonup u_i\pa_x p_i(u)
	\quad\mbox{weakly in }L^{1}(Q_T).
\end{equation*}
In fact, since $u_{m,i}^{1/2}\pa_x^m p_{m,i}$ is uniformly bounded in $L^2(Q_T)$
and $u_{m,i}^{1/2}$ is uniformly bounded in $L^\infty(0,T;L^4(\torus))$, this
weak convergence even holds in $L^2(0,T;L^{4/3}(\torus))$.
This proves \eqref{4.F30} and ends the proof.
\end{proof}

\begin{lemma}\label{lem.aux2}
Let the assumptions of Theorem \ref{thm.conv} hold. Then it 
holds that, as $m\to\infty$, 
$$
  F_{10}^m-F_1^m\to 0, \quad F_{20}^m-F_2^m\to 0, \quad F_{30}^m-F_3^m\to 0.
$$
\end{lemma}

The lemma implies that
\begin{align*}
  F_{10}^m+F_{20}^m+F_{30}^m 
	&= (F_{10}^m-F_1^m) + (F_{20}^m-F_2^m) + (F_{30}^m-F_3^m)
	+ (F_1^m+F_2^m+f_3^m) \\
	&= (F_{10}^m-F_1^m) + (F_{20}^m-F_2^m) + (F_{30}^m-F_3^m) \to 0 \quad 
	\mbox{as }m \to \infty.
\end{align*}
Therefore, thanks to Lemma \ref{lem.aux1}, we conclude that $u=(u_1,\ldots,u_n)$ 
is a weak solution to \eqref{1.eq}--\eqref{1.ic}. This finishes the proof of 
Theorem \ref{thm.conv} once Lemma \ref{lem.aux2} is proved.

\begin{proof}[Proof of Lemma \ref{lem.aux2}]
The limit $F_{10}^m-F_1^m\to 0$ is shown in \cite[Theorem 5.2]{CLP03}.
For the convergence of $F_{20}^m-F_2^m$, we use discrete integration by parts:
\begin{align*}
  F_2^m &= \sigma\sum_{k=1}^{N_T^m}\Delta t_m\sum_{\ell\in G_m}
	\frac{u_{i,\ell+1}^k-u_{i,\ell}^k}{\Delta x_m}
	(\psi_{i,\ell+1}^{k-1}-\psi_{i,\ell}^{k-1}) \\
	&= \sigma\sum_{k=1}^{N_T^m}\sum_{\ell\in G_m}\int_{x_\ell}^{x_{\ell+1}}
	\frac{u_{i,\ell+1}^k-u_{i,\ell}^k}{\Delta x_m}
	\int_{t_{k-1}}^{t_k}\frac{\psi_{i,\ell+1}^{k-1}-\psi_{i,\ell}^{k-1}}{\Delta x_m}
	dxdt, \\
	F_{20}^m 
	&= \sigma\sum_{k=1}^{N_T^m}\sum_{\ell\in G_m}\int_{t_{k-1}}^{t_k}
	\int_{x_\ell}^{x_{\ell+1}}\frac{u_{i,\ell+1}^k-u_{i,\ell}^k}{\Delta x_m}
	\pa_x\psi_i dxdt.
\end{align*}
By the mean-value theorem,
$$
  \bigg|\int_{t_{k-1}}^{t_k}\frac{1}{\Delta x_m}\int_{x_\ell}^{x_{\ell+1}}
	\bigg(\frac{\psi_{i,\ell+1}^{k-1}-\psi_{i,\ell}^{k-1}}{\Delta x_m}-\pa_x\psi_i\bigg)
	dxdt\bigg| \le C\Delta t_m\eta_m.
$$
This shows that, as $m\to\infty$,
\begin{align*}
  |F_2^m&-F_{20}^m| \le \sigma\sum_{k=1}^{N_T^m}\sum_{\ell\in G_m}
	\bigg|\int_{t_{k-1}}^{t_k}\int_{x_\ell}^{x_{\ell+1}}
	\bigg(\frac{\psi_{i,\ell+1}^{k-1}-\psi_{i,\ell}^{k-1}}{\Delta x_m}-\pa_x\psi_i\bigg)
	\frac{u_{i,\ell+1}^k-u_{i,\ell}^k}{\Delta x_m}dxdt\bigg| \\
	&\le C \eta_m\sum_{k=1}^{N_T^m}\Delta t_m\sum_{\ell\in G_m}
	|u_{i,\ell+1}^k-u_{i,\ell}^k| 
	= C\eta_m\sum_{k=1}^{N_T^m}\Delta t_m|u_i^k|_{1,1,\T_m}\to 0,
\end{align*}
where we used the uniform discrete $L^2(0,T;W^{1,1}(\torus))$ bound 
from Lemma \ref{lem.est}.

It remains to prove that $|F_{30}^m-F_{3}^m|\to 0$. First, using a discrete integration by parts we rewrite $F_3^m$ as well as $F_{30}^m$ as
\begin{align}
  F_3^m &= \sum_{k=1}^{N_T^m}\sum_{\ell\in G_m}\int_{t_{k-1}}^{t_k}
	u_{i,\ell+1/2}^k\frac{p_{i,\ell+1}^k-p_{i,\ell}^k}{\Delta x_m}
	(\psi_{i,\ell+1}^{k-1}-\psi_{i,\ell}^{k-1})dt, \nonumber \\
	F_{30}^m &= \sum_{k=1}^{N_T^m}\sum_{\ell\in G_m}\int_{t_{k-1}}^{t_k}\bigg(
	\int_{x_\ell}^{x_{\ell+1/2}}u_{i,\ell}^k
	\frac{p_{i,\ell+1}^k-p_{i,\ell}^k}{\Delta x_m}\pa_x\psi_i dx \nonumber \\
	&\phantom{xx}{}+ \int_{x_{\ell+1/2}}^{x_{\ell+1}}u_{i,\ell+1}^k
	\frac{p_{i,\ell+1}^k-p_{i,\ell}^k}{\Delta x_m}\pa_x\psi_i dx\bigg) \nonumber 
\end{align}
Then we find that
\begin{align*}
  |F_3^m-F_{30}^m| &= \bigg|\sum_{k=1}^{N_T^m} \sum_{\ell \in G_m} 
  (u^k_{i,\ell+1/2}-u^k_{i,\ell})\frac{p^k_{i,\ell+1}-p^k_{i,\ell}}{\Delta x_m} \\
  &\phantom{xxxx}{}\times\int_{t_{k-1}}^{t_k} 
	\bigg(\frac{\psi_{i,\ell+1}^{k-1}-\psi_{i,\ell}^{k-1}}{2} - \int_{x_\ell}^{x_{\ell+1/2}} 
	\pa_x \psi_i(x) dx \bigg)dt \\
  &\phantom{xx}{}+ \sum_{k=1}^{N_T^m} \sum_{\ell \in G_m}
	(u^k_{i,\ell+1/2}-u^k_{i,\ell+1})\frac{p^k_{i,\ell+1}-p^k_{i,\ell}}{\Delta x_m} \\
	&\phantom{xxxx}{}
	\times\int_{t_{k-1}}^{t_k} \bigg(\frac{\psi_{i,\ell+1}^{k-1}-\psi_{i,\ell}^{k-1}}{2}
	- \int_{x_{\ell+1/2}}^{x_{\ell+1}} \pa_x \psi_i(x) dx\bigg) dt \bigg|,
\end{align*}
Thanks to the regularity of $\psi_i$, there exists a constant $C$ 
independent of $\eta_m$ such that
\begin{align*}
  \bigg|\int_{t_{k-1}}^{t_k} \bigg(\frac{\psi_{i,\ell+1}^{k-1}-\psi_{i,\ell}^{k-1}}{2}
	- \int_{x_\ell}^{x_{\ell+1/2}} \pa_x \psi_i(x) dx \bigg) dt\bigg| 
	\leq C \eta_m \Delta t_m.
\end{align*}
We obtain a similar expression if we integrate $\pa_x \psi_i$ over $(x_{\ell+1/2},x_{\ell+1})$. Thus, since
\begin{align*}
  |u_{i,\ell+1/2}^k-u_{i,\ell}^k|
	&\le |u_{i,\ell+1}^k-u_{i,\ell}^k|\quad\mbox{and} \\
  |u_{i,\ell+1/2}^k-u_{i,\ell+1}^k|
	&\le |u_{i,\ell}^k-u_{i,\ell+1}^k|,
\end{align*}
we have
\begin{align*}
  |F_3^m-F_{30}^m| &\leq 2 C \eta_m \sum_{k=1}^{N_T^m} \Delta t_m \sum_{\ell \in G_m} 
	|u^k_{i,\ell+1}-u^k_{i,\ell}||\mathrm{D}_\ell \, p^k_i| \\
  &\leq 2 C \eta_m \bigg(\sum_{i=1}^n a_{ii} \sum_{k=1}^{N^m_T} \Delta t_m 
	|u^k_i|^2_{1,2,\T_m} \\
  &\phantom{xx}+ \sum_{\substack{j=1\\j\neq i}}^n \sum_{k=1}^{N_T^m} \Delta t_m 
	\sum_{\ell, \ell' \in G_m} |u^k_{i,\ell+1}-u^k_{i,\ell}| |a_{ij} 
	(B^{ij}_{\ell+1-\ell'}-B^{ij}_{\ell-\ell'}) u^k_{j,\ell'}|\bigg).
\end{align*}
It follows for $j \in \lbrace 1,\ldots,n \rbrace$ with $j \neq i$, using the discrete analog \eqref{2.B} of $\pa_x B^{ij}*u_j=B^{ij}*\pa_x u_j$, that
\begin{align}\nonumber
  \max_{\ell \in G_m} \bigg(\sum_{ \ell' \in G_m} |a_{ij} 
	(B^{ij}_{\ell+1-\ell'}-B^{ij}_{\ell-\ell'}) u^k_{j,\ell'}|\bigg)	
	&= \max_{\ell \in G_m} \bigg(\sum_{\ell'\in G_m}\Delta x_m |a_{ij}| |B^{ij}_{\ell-\ell'}| 
	|\mathrm{D}_{\ell'}u_j^k| \bigg) \nonumber \\
	&\leq |a_{ij}|\|B^{ij}\|_{L^\infty(\torus)} |u^k_j|_{1,1,\T_m}. \nonumber 
\end{align}
At this point, we need the regularity condition $B^{ij}\in L^\infty(\torus)$
from Hypothesis (H3). Hence, it holds that
\begin{align*}
  |F_3^m-F_{30}^m| \leq 2 C \eta_m \bigg(\sum_{i=1}^n \sum_{k=1}^{N^m_T} \Delta t_m 
	|u^k_i|^2_{1,2,\T_m} + \sum_{k=1}^{N_T^m} \Delta t_m |u^k_i|_{1,1,\T_m} 
	\sum_{\substack{j=1\\j\neq i}}^n  |u^k_j|_{1,1,\T_m} \bigg).
\end{align*}
It remains to apply the Cauchy--Schwarz inequality to conclude that
\begin{align*}
  |F_3^m-F_{30}^m| &\leq 2 C \eta_m \bigg\{\sum_{i=1}^n \sum_{k=1}^{N^m_T} \Delta t_m 
	|u^k_i|^2_{1,2,\T_m} \\
	&\phantom{xx}{}+ \sum_{\substack{j=1\\j\neq i}}^n \bigg(\sum_{k=1}^{N_T^m} \Delta t_m 
	|u^k_i|^2_{1,1,\T_m}\bigg)^{1/2} \bigg(\sum_{k=1}^{N_T^m} \Delta t_m |u^k_j|^2_{1,1,\T_m}
	\bigg)^{1/2} \bigg\}.
\end{align*}
Finally, we infer from Lemma \ref{lem.est} that $|F_3^m-F_{30}^m|\to 0$ as $m\to\infty$. 
Here, we need the discrete $L^2(0,T;H^1(\torus))$ bound for $u_i$, at least of $a_{ii}>0$.
This concludes the proof of Lemma \ref{lem.aux2}.
\end{proof}


\section{Numerical experiments}\label{sec.num}

In this section, we present several numerical experiments to illustrate the behavior of the scheme. The scheme was implemented in one space dimension using Matlab. In all the subsequent numerical tests, we choose the upwind mobility \eqref{2.upwind}. The code is available at \url{https://gitlab.tuwien.ac.at/asc/nonlocal-crossdiff}. Our code is an adaptation of that one developed in \cite{HeZu22} for the approximation of the nonlocal SKT system. We refer the reader to \cite[Section 6.1]{HeZu22} for a complete presentation of the different methods used to implement the scheme.

\subsection{Test case 1. Rate of convergence in space for various $L^p$-norms, convolution kernels, and initial data}

We investigate the rate of convergence in space of the scheme at final time $T=1$. In all test cases of this section, we consider $n=2$ species, $\sigma = 10^{-4}$, the coefficient matrix $A=(a_{ij})_{1\leq i,j \leq 2}$ given by 
$$
A = \begin{pmatrix} 0.1251 & 0.25 \\ 1 & 2 \end{pmatrix},
$$
and $\pi_1 = 4$, $\pi_2 = 1$. We consider various initial data and kernels. More precisely, we choose
\begin{align}\label{1.init:non-smooth}
	& u_1^0(x) = \mathbf{1}_{[1/4,3/4]}(x), \quad 	u_2^0(x) = \mathbf{1}_{[0,1/4]}(x) + \mathbf{1}_{[3/4,1]}(x), \\
    \label{1.init:smooth}
	& u_1^0(x) = \cos \left( 2 \pi x \right) + 1, \quad
	u_2^0(x) = \sin \left( 2 \pi x - \pi/2 \right) + 1, \\
    \label{1.init:continuous}
	& u_1^0(x) = \max \left( 1 - |1-2x|, 0 \right), \quad
	u_2^0(x) = \max \left( 1 - 2|x|, 0 \right)
\end{align}
and the kernels
\begin{align}\label{1.kernel:indicator}
	B^{ij}(z) &= \mathbf{1}_{[-0.3,0.3]}(z), \\
    \label{1.kernel:triangle}
	B^{ij}(z) &= 2\max \left( 1- |z|/0.3, 0 \right), \\
    \label{1.kernel:Gaussian}
	B^{ij}(z) &= \exp \left( -|z|^2/2 \varepsilon^2 \right)/\sqrt{2 \pi \varepsilon^2}, ~ \varepsilon = 10^{-3}.
\end{align}

First, we consider a mesh of $N_{init} = 32$ cells and the time step size $\Delta t_{init} = 1/64$. Then, starting from this initial mesh, we refine the mesh in space by doubling the number of cells and halving the time step size, i.e.\ $N_{\rm new} = 2N_{\rm old}$ and $\Delta t_{\rm new} = \Delta t_{\rm old}/2$. This refinement of the meshes is in agreement with the first-order convergence rate of the Euler discretization in time and the expected first-order convergence rate in space of the scheme, due to the choice of the upwind mobility in the numerical fluxes. As exact solutions to system \eqref{1.eq}--\eqref{1.p} are not explicitly known, we refine the mesh in space and time until $N_{\rm end} = 2048$ and $\Delta t_{\rm end} = 1 / 4096$, and we consider the solutions of the scheme obtained for $N_{\rm end}$ and $\Delta t_{\rm end}$ as reference solutions. The error is computed between the reference solutions and the solutions obtained for $N = 1024$ cells  and $\Delta t = 1/2048$ at final time $T=1$. Finally, using linear regression in logarithmic scale, we present in Table \ref{tab:OrderConvSpace} the experimental order of convergence in the $L^1$ and $L^\infty$-norms. As expected, we observe a rate of convergence around one. In Table~\ref{tab:OrderConvSpace}, the numbers in bold letters denote the number of the test case available in our code (see the file \textsf{loadTestcase.m}).

\begin{table}
\begin{tabular}{ c|r l|r l|r l }
	Kernel $\rightarrow$ & \multicolumn{2}{c|}{\multirow{3}{*}{Indicator \eqref{1.kernel:indicator}}} & \multicolumn{2}{c|}{\multirow{3}{*}{Triangle \eqref{1.kernel:triangle}}} & \multicolumn{2}{c}{\multirow{3}{*}{Gaussian \eqref{1.kernel:Gaussian}}}  \\
	\cline{1-1} & & &  \\
	Initial Data $\downarrow$ & & & & & &  \\
	\hline
	\hline
	\multirow{6}{*}{ \eqref{1.init:non-smooth}} & \multicolumn{2}{c|}{{\bf Testcase 13}} & \multicolumn{2}{c|}{{\bf Testcase 16}} & \multicolumn{2}{c}{{\bf Testcase 19}}  \\
	 & $L^1$-order: & $1.1741$ & $L^1$-order: & $1.1741$ & $L^1$-order: & $1.0109$  \\
	 & $L^1$-error: & $9.76\cdot10^{-4}$ & $L^1$-error: & $9.76\cdot10^{-4}$ & $L^1$-error: & $3.20\cdot10^{-3}$  \\
	 & $L^\infty$-order: & $1.14$ & $L^\infty$-order: & $1.1331$ & $L^\infty$-order: & $0.98437$  \\
	 & $L^\infty$-error: & $1.49\cdot10^{-3}$ & $L^\infty$-error: & $1.68\cdot10^{-3}$ & $L^\infty$-error: & $2.45\cdot10^{-2}$  \\
	 \hline
	 \multirow{6}{*}{ \eqref{1.init:smooth}} & \multicolumn{2}{c|}{{\bf Testcase 14}} & \multicolumn{2}{c|}{{\bf Testcase 17}} & \multicolumn{2}{c}{{\bf Testcase 20}}  \\
	 & $L^1$-order: & $1.0948$ & $L^1$-order: & $1.0336$ & $L^1$-order: & $0.93381$  \\
	 & $L^1$-error: & $1.81\cdot10^{-5}$ & $L^1$-error: & $2.78\cdot10^{-5}$ & $L^1$-error: & $2.35\cdot10^{-3}$  \\
	 & $L^\infty$-order: & $1.0486$ & $L^\infty$-order: & $1.0092$ & $L^\infty$-order: & $0.91831$  \\
	 & $L^\infty$-error: & $4.73\cdot10^{-5}$ & $L^\infty$-error: & $8.57\cdot10^{-5}$ & $L^\infty$-error: & $8.87\cdot10^{-3}$  \\
	 \hline
	 \multirow{6}{*}{ \eqref{1.init:continuous}} & \multicolumn{2}{c|}{{\bf Testcase 15}} & \multicolumn{2}{c|}{{\bf Testcase 18}} & \multicolumn{2}{c}{{\bf Testcase 21}}  \\
	 & $L^1$-order: & $0.97752$ & $L^1$-order: & $0.97495$ & $L^1$-order: & $0.9611$ \\
	 & $L^1$-error: & $6.39\cdot10^{-5}$ & $L^1$-error: & $5.35\cdot10^{-5}$ & $L^1$-error: & $9.27\cdot10^{-4}$  \\
	 & $L^\infty$-order: & $0.99787$ & $L^\infty$-order: & $0.99741$ & $L^\infty$-order: & $0.9761$ \\
	 & $L^\infty$-error: & $1.74\cdot10^{-4}$ & $L^\infty$-error: & $11.48\cdot10^{-4}$ & $L^\infty$-error: & $3.69\cdot10^{-3}$ 
\end{tabular}
\caption{Orders of convergence in the $L^1$ and $L^\infty$ norms in space at final time $T=1$ for different kernels and initial data.}
\label{tab:OrderConvSpace}
\end{table}


\subsection{Test case 2. Rate of convergence of the localization limit in various metrics}

In the second test case, following \cite{HeZu22}, we evaluate numerically the rate of convergence of the localization limit. More precisely, for some sequences of kernels converging towards the Dirac measure $\delta_0$, we compute the rate of convergence in different metrics of the solutions to scheme \eqref{2.ic}--\eqref{2.fv3} towards its local version, i.e.\ $B^{ij} = \delta_0$ for all $i,j=1,\ldots,n$. At the continuous level, one can show by adapting the approach of \cite{JPZ22} that the localization limit holds thanks to a compactness method; see also \cite{DiMo21} for the SKT system. However, so far no explicit rate of convergence is available. The goal of this numerical test is to obtain a better insight into this rate of convergence. Besides, it also illustrates Remark \ref{rem.ap}.

We consider the following parameters (for all 6 test cases of this section): $n=3$ species, diffusion parameter $\sigma = 10^{-4}$, coefficient matrix 
$$
  A = \begin{pmatrix} 0.5 & 0.2 & 0.125 \\ 0.4 & 1 & 0.2 \\ 0.25 & 0.2 & 1 \end{pmatrix},
$$
and $\pi_1 = 4$, $\pi_2 = 2$, $\pi_3 = 2$. We choose the final time $T = 1$, a mesh of $N=512$ cells, and the time step size $\Delta t = 10^{-3}$. Furthermore, we take the nonsmooth initial data
\begin{equation}\label{2.init:non-smooth}
	u_1^0(x) = \mathbf{1}_{[3/6,5/6]}(x), \quad
	u_2^0(x) = \mathbf{1}_{[0,1/6]}(x) + \mathbf{1}_{[5/6, 1]}(x), \quad
	u_3^0(x) = \mathbf{1}_{[1/6,3/6]}(x),
\end{equation}
and the smooth initial data
\begin{align}\label{2.init:smooth}
	u_1^0(x) &= \cos \left( 2 \pi x \right) + 1, \quad
	u_2^0(x) = \sin \left( 2 \pi x \right) + 1, \\
	&\phantom{}u_3^0(x) = \left( \cos \left( 2 \pi x\right) + \sin \left( 2 \pi x \right) + 2 \right)/2.\nonumber
\end{align}
The kernels are choosen according to
\begin{align}\label{2.kernel:indicator}
	B_\alpha^{ij}(z) &= \mathbf{1}_{[-\alpha, \alpha]}(z)/2 \alpha, \\
    \label{2.kernel:triangle}
	B_\alpha^{ij}(z) &= \max \left( 1 - |z|/\alpha, 0 \right)/\alpha, \\
    \label{2.kernel:Gaussian}
	B_\alpha^{ij}(z) &= \exp \left( - |z|^2/2 \alpha^2 \right)/\sqrt{2 \pi \alpha^2}.
\end{align}

In our experiments, starting from $\alpha_{\rm init} = 2^7  \Delta x$, we successively halve $\alpha$ until we reach $\alpha=\Delta x$. For each value of $\alpha$ we compute the solutions to the nonlocal scheme \eqref{2.ic}--\eqref{2.fv3} at final time. We evaluate the $L^1$, $L^\infty$, and Wasserstein distance $W_1$ between the solution to the nonlocal scheme and the solution to the local one (for this, it is enough to set $\alpha=0$ in our code). Since we are in one space dimension, we can explicitly compute the Wasserstein distance $W_1$; see \cite[Chapter 2]{Santam15}. The rates of convergence are estimated by linear regression (in log scale) and the results are presented in Table \ref{tab:Nonlocal-Local}. Surprisingly, we observe a slightly better rate of convergence in the case of nonsmooth initial data. As before, in Table~\ref{tab:Nonlocal-Local}, the names in bold letters denote the name of the test case available in our code (see the file \textsf{loadTestcase.m}).

\begin{table}[h!]
\begin{tabular}{c|r l|r l|r l}
	Kernel $\rightarrow$ & \multicolumn{2}{c|}{\multirow{3}{*}{ \eqref{2.kernel:indicator}}} & \multicolumn{2}{c|}{\multirow{3}{*}{ \eqref{2.kernel:triangle}}} & \multicolumn{2}{c}{\multirow{3}{*}{ \eqref{2.kernel:Gaussian}}} \\
	\cline{1-1} & & & & & & \\
	Initial Data $\downarrow$ & & & & & & \\
	\hline
	\hline
	\multirow{4}{*}{nonsmooth \eqref{2.init:non-smooth}} & \multicolumn{2}{c|}{{\bf Testcase NLTL2}} & \multicolumn{2}{c|}{{\bf Testcase NLTL4}} & \multicolumn{2}{c}{{\bf Testcase NLTL6}} \\
	& $L^1$-order: & $1.8280$ & $L^1$-order: & $1.8709$ & $L^1$-order: & $1.7386$ \\
	& $L^\infty$-order: & $1.8271$ & $L^\infty$-order: & $1.8698$ & $L^\infty$-order: & $1.7379$ \\
	& $W_1$-order: & $1.8306$ & $W_1$-order: & $1.8724$ & $W_1$-order: & $1.7426$ \\
	\hline
	\multirow{4}{*}{smooth \eqref{2.init:smooth}} & \multicolumn{2}{c|}{{\bf Testcase NLTL3}} & \multicolumn{2}{c|}{{\bf Testcase NLTL5}} & \multicolumn{2}{c}{{\bf Testcase NLTL7}} \\
	& $L^1$-order: & $1.7430$ & $L^1$-order: & $1.8240$ & $L^1$-order: & $1.5991$ \\
	& $L^\infty$-order: & $1.7462$ & $L^\infty$-order: & $1.8261$ & $L^\infty$-order: & $1.6038$ \\
	& $W_1$-order: & $1.7451$ & $W_1$-order: & $1.8252$ & $W_1$-order: & $1.6023$
\end{tabular}
\caption{Rates of convergence of the localization limit in the $L^1$, $L^\infty$ and $W_1$ metric for different initial data and kernels.}
\label{tab:Nonlocal-Local}
\end{table}

\subsection{Test case 3. Segregation phenomenon}

In the last numerical experiment, we set $\sigma = 0$. Under the assumptions $n=2$ species, $a_{ij}=1$, and $B^{ij} = \delta_0$ for $i,j=1,2$, it has been shown in \cite{BGHP85} that if the initial data are segregated (initial data with disjoint supports) then the solutions remain segregated for all time. The main goal of this subsection is to illustrate the segregation pattern due to the nonlocal terms, i.e.\ $B^{ij} \neq \delta_0$. Let us notice that in the subsequent test cases, Hypothesis (H3) is never satisfied. However, we did not encounter any numerical issues with our code. 

We launched the code for a mesh of $512$ cells and the time step size $\Delta t = 10^{-4}$. In the case of $n=2$ species, we considered the initial data
\begin{equation*}
	u_1^0(x) = \mathbf{1}_{[0.1,0.4]}(x), \quad
	u_2^0(x) = \mathbf{1}_{[0.6,0.8]}(x),
\end{equation*}
while for $n=3$ species, we have taken
\begin{equation*}
	u_1^0(x) = \mathbf{1}_{[0.5,0.6]}(x), \quad
	u_2^0(x) = \mathbf{1}_{[0.8,0.9]}(x), \quad
	u_3^0(x) = \mathbf{1}_{[0.1,0.2]}(x).
\end{equation*}
In both cases, we set $a_{ij}=1$ for all $i,j=1,\ldots,n$.

In Figures \ref{fig1} and \ref{fig2}, we present the segregation pattern at time $t=0.02$ and $t=0.2$ obtained for the local model, $B^{ij}=\delta_0$, and the nonlocal model with
\begin{equation*}
	B^{ij}(z) = 100\cdot \mathbf{1}_{[-0.1,0.1]}(z).
\end{equation*}
For small times, the support of the species extends until reaching the support of another species. In the local model, the species slightly mix (due to numerical diffusion), while we observe a ``gap'' between the supports of the solutions in the nonlocal model. This ``gap'' is of order $0.1$ which is the size of the radius of the kernels $B^{ij}$. Similar numerical results have been observed in \cite[Section 6]{CFS20} but using different kernel functions and two species only.

\begin{figure}[ht!]
\includegraphics[width=73mm,height=60mm]{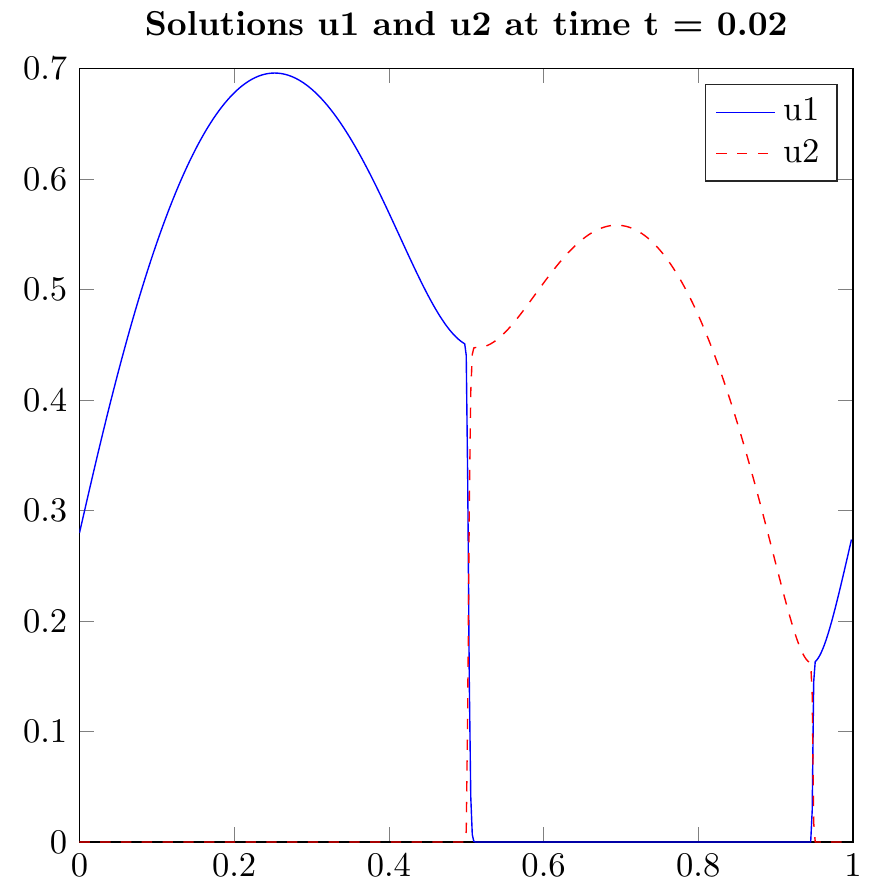}
\includegraphics[width=73mm,height=60mm]{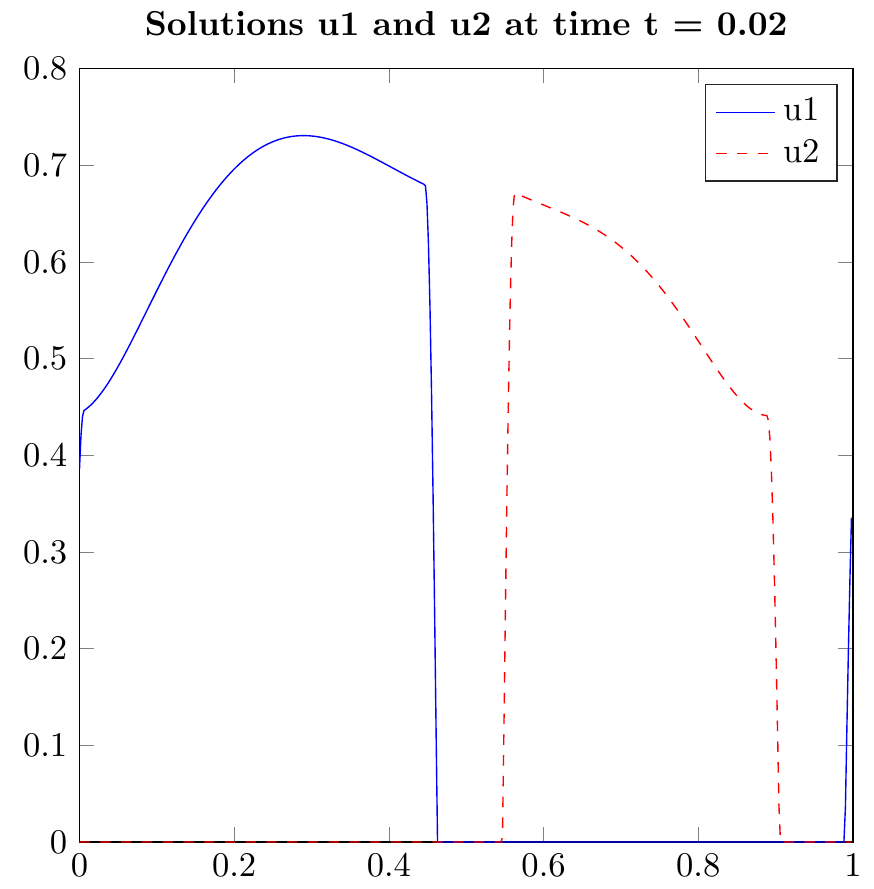}
\includegraphics[width=73mm,height=60mm]{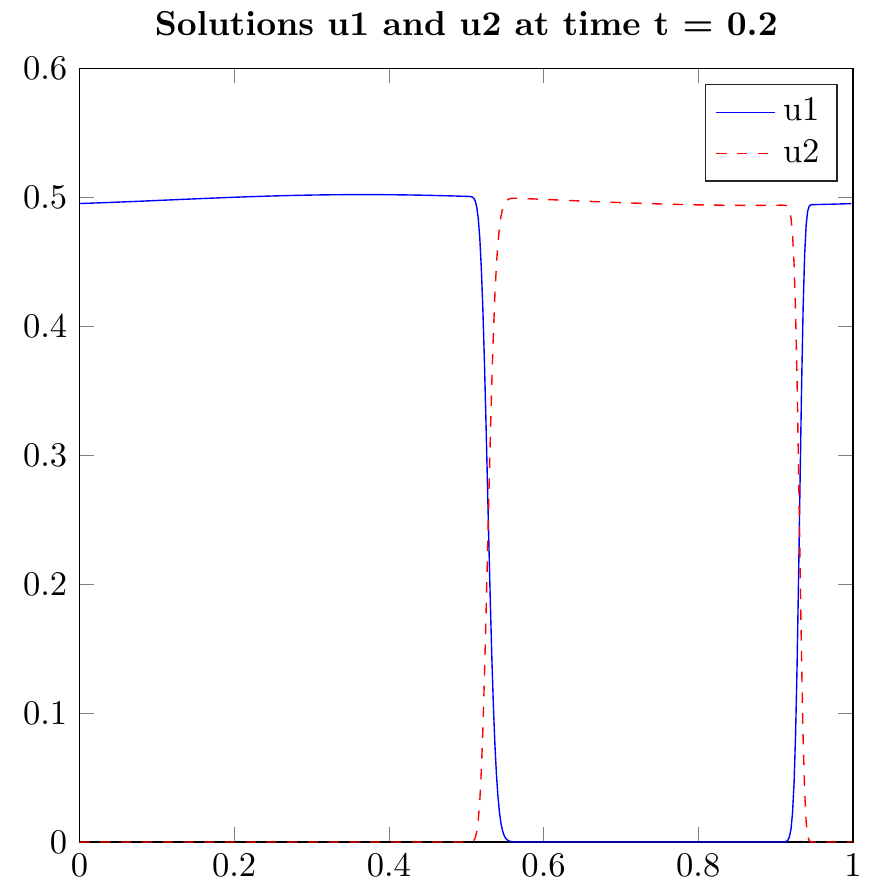}
\includegraphics[width=73mm,height=60mm]{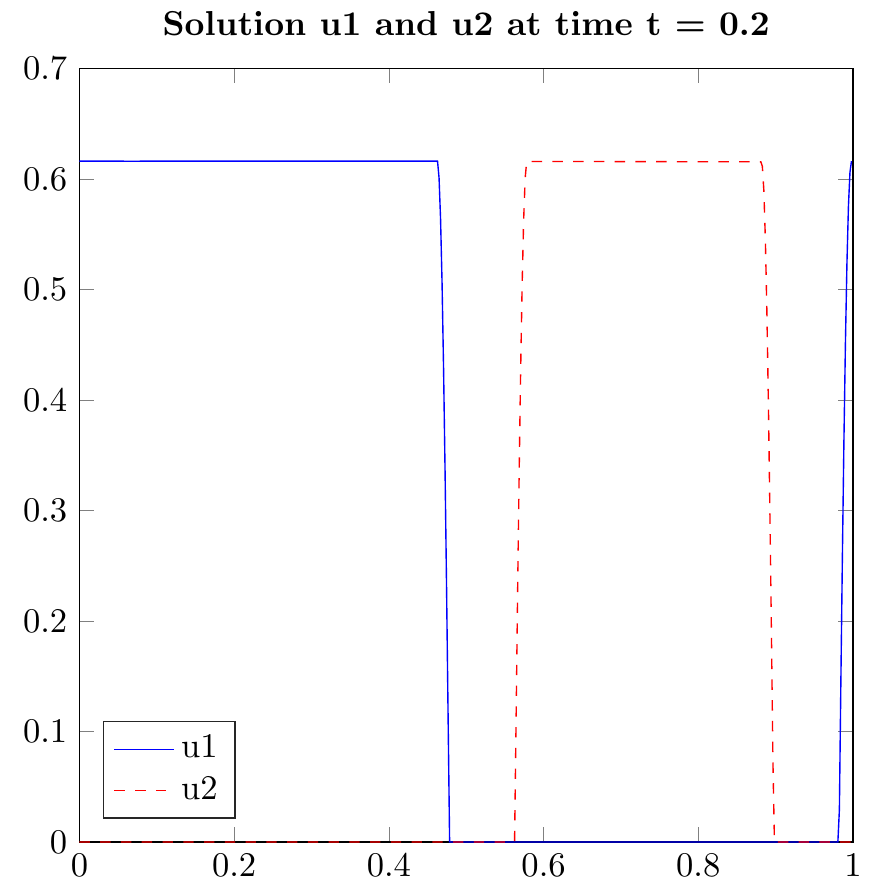}
\caption{Comparison of the segregation pattern for {\em two} species at times $t=0.02$ (top) 
and $t=0.2$ (bottom) obtained from the local model (left) and nonlocal model (right). The solutions are almost in the steady state at $t=0.2$.}\label{fig1}
\end{figure}

\begin{figure}[ht!]
\includegraphics[width=73mm,height=60mm]{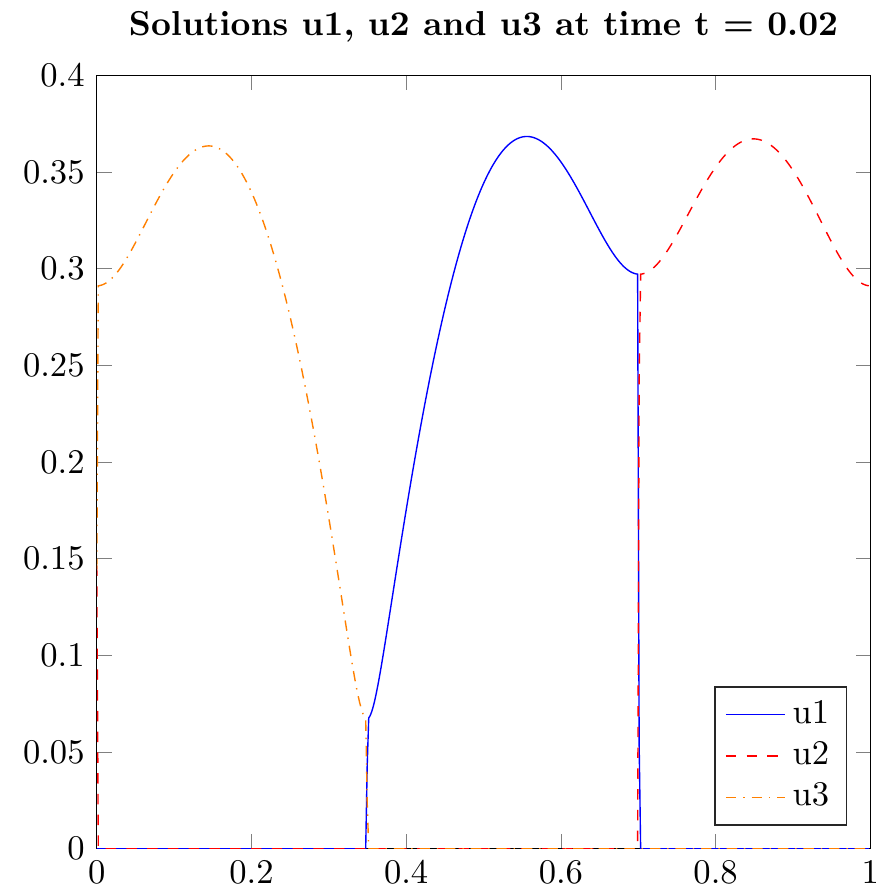}
\includegraphics[width=73mm,height=60mm]{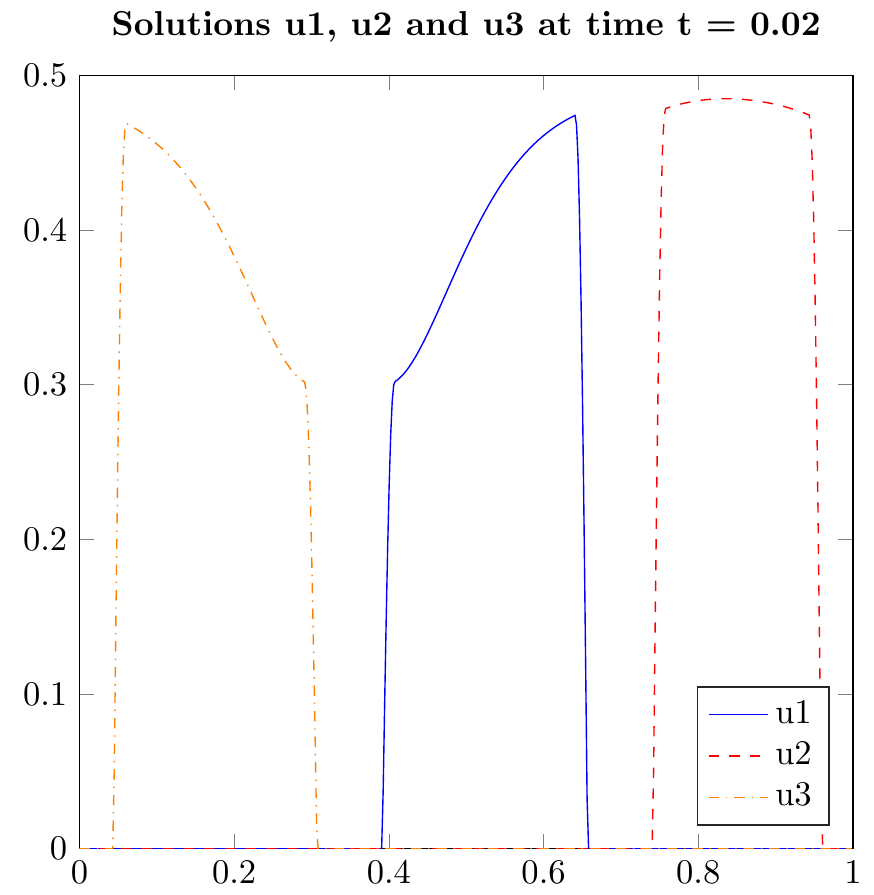}
\includegraphics[width=73mm,height=60mm]{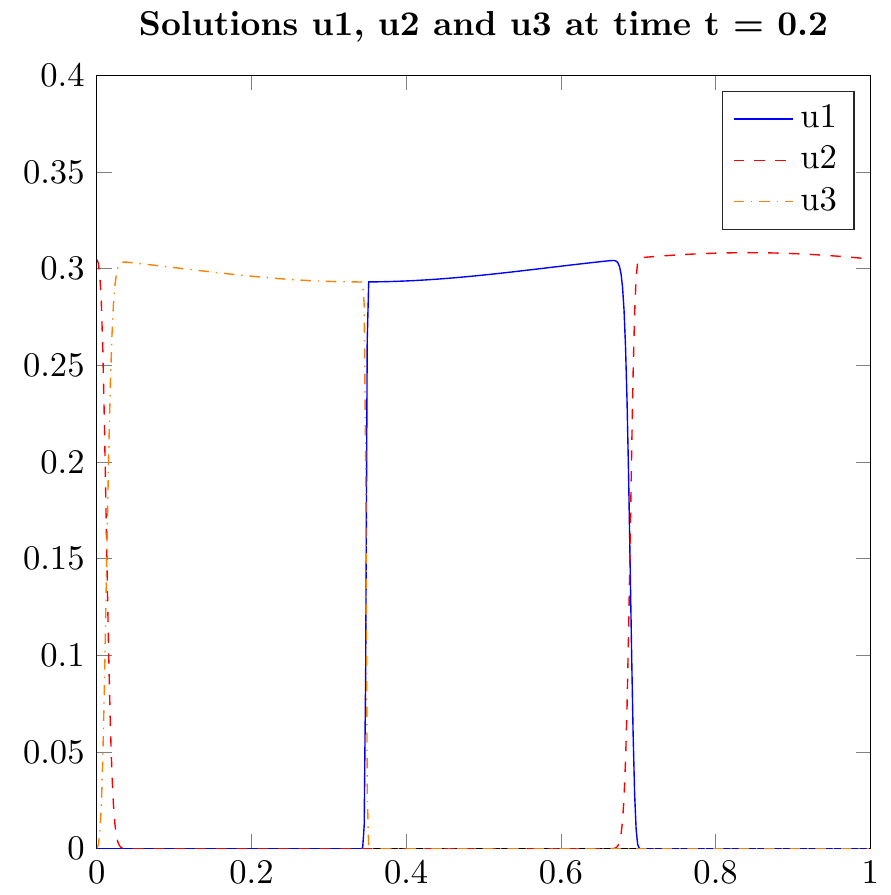}
\includegraphics[width=73mm,height=60mm]{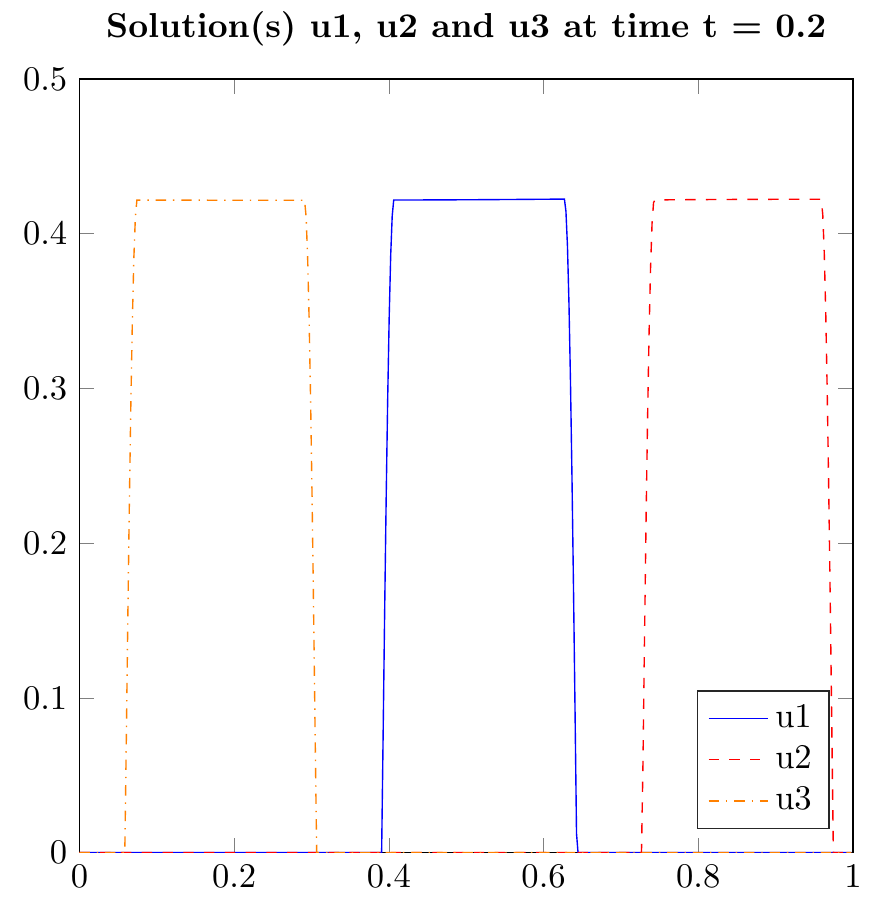}
\caption{Comparison of the segregation patterns for {\em three} species at times $t=0.02$ (top) 
and $t=0.2$ (bottom) obtained from the local model (left) and nonlocal model (right).  The solutions are almost in the steady state at $t=0.2$.}\label{fig2}
\end{figure}


\begin{appendix}
\section{Some auxiliary results}\label{sec.aux}

\begin{lemma}\label{lem.Q}
Under Hypothesis (H3), the entropy dissipation $Q$, defined in \eqref{1.Q}, is nonnegative.
\end{lemma}

\begin{proof}
We follow the approach of \cite{DiMo21} and write $Q=Q_1+\cdots+Q_3$, where
\begin{align*}
  Q_1 &= \frac{1}{n-1}\sum_{i,j=1,\,i<j}^n \int_\torus \pi_i a_{ii} |\pa_x u_i(x)|^2 dx
	+ \frac{1}{n-1}\sum_{i,j=1,\,i>j}^n \int_\torus \pi_i a_{ii} |\pa_x u_i(y)|^2 dy, \\
  Q_2 &= \sum_{i,j=1,\,i<j}^n \int_\torus \int_{\torus}\pi_i a_{ij} B^{ij}(x-y) 
	\pa_x u_j(y)\pa_x u_i(x) dydx, \\
	Q_3 &= \sum_{i,j=1,\,i>j}^n \int_\torus \int_{\torus}\pi_i a_{ij} B^{ij}(x-y) 
	\pa_x u_j(y)\pa_x u_i(x) dydx.
\end{align*}
Exchanging $i$ and $j$ in the second integral of $Q_1$ and using $\m(\torus)=1$, we have
$$
  Q_1 = \frac{1}{n-1}\sum_{i,j=1,\,i<j}^n\int_\torus \int_\torus\big(
	\pi_i a_{ii} |\pa_x u_i(x)|^2 dx + \pi_j a_{jj} |\pa_x u_j(y)|^2\big)dydx.
$$
Exchanging $i$ and $j$ as well as $x$ and $y$ in $Q_3$ gives
\begin{align*}
  Q_3 &= \sum_{i,j=1,\,i<j}^n\int_\torus \int_{\torus}\pi_j a_{ji} B^{ji}(y-x)
	\pa_x u_j(y)\pa_x u_i(x) dydx \\
	&= \sum_{i,j=1,\,i<j}^n\int_\torus \int_{\torus}\pi_j a_{ji} B^{ij}(x-y)
	\pa_x u_j(y)\pa_x u_i(x) dydx.
\end{align*}
We collect these expressions to obtain
\begin{align*}
  Q = \frac{1}{(n-1)} \sum_{i,j=1,\,i < j}^n \int_\torus \int_\torus 
  \begin{pmatrix} \pa_x u_i(x) \\ \pa_x u_j(y) \end{pmatrix}^\top
  M^{ij}(x-y)
  \begin{pmatrix} \pa_x u_i(x) \\ \pa_x u_j(y)\end{pmatrix} 
	dydx \geq 0,
\end{align*}
where $M^{ij}$ is defined in \eqref{1.M}, and the last inequality follows from
Hypothesis (H3). 
\end{proof}

\begin{lemma}\label{lem.prop}
The upwind approximation \eqref{2.upwind} and the logarithmic mean \eqref{2.logmean}
satisfy property \eqref{2.prop} of the mobilities $u_{i,\sigma}$.
\end{lemma}

\begin{proof}
The proof is based on the following inequalities for the logarithmic mean:
\begin{equation}\label{a.logmean}
  \min\{a,b\}\le \frac{a-b}{\log a-\log b}\le\max\{a,b\}\quad\mbox{for all }a,b>0.
\end{equation}
They imply the linear growth $u_{i,\ell+1/2}\le\max\{u_{i,\ell},u_{i,\ell+1}\}$
for the logarithmic mean, which also holds, by definition, 
for the upwind approximation.
We show that property \eqref{2.prop} is satisfied for the
upwind approximation \eqref{2.upwind}. 
Let $p_{i,\ell+1}-p_{i,\ell}\ge 0$. Then, by \eqref{a.logmean},
\begin{align*}
  u_{i,\ell+1/2}(p_{i,\ell+1}-p_{i,\ell})(\log u_{i,\ell+1}-\log u_{i,\ell})
	&= u_{i,\ell+1}(p_{i,\ell+1}-p_{i,\ell})(\log u_{i,\ell+1}-\log u_{i,\ell}) \\
  &\ge (p_{i,\ell+1}-p_{i,\ell})(u_{i,\ell+1}-u_{i,\ell}).
\end{align*}
On the other hand, if $p_{i,\ell+1}-p_{i,\ell}<0$, again by \eqref{a.logmean},
\begin{align*}
  u_{i,\ell+1/2}(p_{i,\ell+1}-p_{i,\ell})(\log u_{i,\ell+1}-\log u_{i,\ell})
	&= u_{i,\ell}(p_{i,\ell+1}-p_{i,\ell})(\log u_{i,\ell+1}-\log u_{i,\ell}) \\
	&\ge (p_{i,\ell+1}-p_{i,\ell})(u_{i,\ell+1}-u_{i,\ell}).
\end{align*}
Property \eqref{2.prop} follows immediately after inserting definition 
\eqref{2.logmean} of the logarithmic mean.
This ends the proof.
\end{proof}

\begin{lemma}[Discrete Young convolution inequality]\label{lem.young}
Let $1\le p,q\le\infty$ and $1\le r\le\infty$ 
be such that $1+1/r=1/p+1/q$ and let $B\in L^p(\torus)$
and $v=(v_\ell)_{\ell\in G}\in\mathcal{V}_\T$. Furthermore, let $B_{\ell-\ell'} = (\Delta x)^{-1}\int_{K_{\ell-\ell'}}B(y)dy$ for every $\ell$ and $\ell' \in G$. Then
$$
  \bigg(\sum_{\ell\in G}\Delta x \bigg|\sum_{\ell'\in G}\Delta x
	B_{\ell-\ell'} \, v_{\ell'} \bigg|^r\bigg)^{1/r}
	\le \|B\|_{L^p(\torus)}\|v\|_{0,q,\T}.
$$
\end{lemma}

\begin{proof}
First, let $\ell \in G$ be fixed. Then
\begin{align*}
  \bigg|\sum_{\ell' \in G} \Delta x B_{\ell-\ell'} v_{\ell'} \bigg| 
	\leq \sum_{\ell'\in G} \Delta x \big(|B_{\ell-\ell'}|^p|v_{\ell'}|^q \big)^{1/r}
	|B_{\ell-\ell'}|^{(r-p)/r}|v_{\ell'}|^{(r-q)/r}.
\end{align*}
Thanks to the assumption $1=1/p+1/q-1/r$, we can apply H\"older's inequality with 
exponents $r$, $pr/(r-p)$, and $qr/(r-q)$ to obtain
\begin{align*}
  \bigg|\sum_{\ell' \in G} \Delta x B_{\ell-\ell'}v_{\ell'} \bigg| 
	&\leq \bigg(\sum_{\ell' \in G} \Delta x |B_{\ell-\ell'}|^p |v_{\ell'}|^q \bigg)^{1/r}  
	\bigg(\sum_{\ell' \in G} \Delta x|B_{\ell-\ell'}|^p\bigg)^{(r-p)/pr} \\
	&\phantom{xx}{}\times\bigg(\sum_{\ell' \in G} \Delta x |v_{\ell'}|^q\bigg)^{(r-q)/qr} \\
  &= \bigg(\sum_{\ell' \in G} \Delta x |B_{\ell-\ell'}|^p |v_{\ell'}|^q \bigg)^{1/r} 
	\|B\|^{(r-p)/r}_{0,p,\T}\|v\|^{(r-q)/r}_{0,q,\T}.
\end{align*}
Then, taking the exponent $r$ and summing over $\ell\in G$,
\begin{align*}
  \sum_{\ell \in G} \Delta x \bigg|\sum_{\ell' \in G} \Delta x B_{\ell-\ell'} 
	v_{\ell'} \bigg|^r 
	&\leq \|B\|^{r-p}_{0,p,\T} \|v\|^{r-q}_{0,q,\T} \bigg(\sum_{\ell \in G} \Delta x 
	\sum_{\ell' \in G} \Delta x |B_{\ell-\ell'}|^p |v_{\ell'}|^q \bigg) \\
  &\leq \|B\|^{r-p}_{0,p,\T}\|v\|^{r-q}_{0,q,\T} \bigg(\sum_{\ell' \in G} \Delta x 
	|v_{\ell'}|^q \sum_{\ell \in G} \Delta x|B_{\ell-\ell'}|^p \bigg) \\
  &\leq \|B\|^{r-p}_{0,p,\T} \|v\|^{r-q}_{0,q,\T} \|v\|^q_{0,q,\T} \|B\|^p_{0,p,\T}
  = \|B\|^r_{0,p,\T}\|v\|^r_{0,q,\T}.
\end{align*}
Finally, it holds that
\begin{align*}
  \|B\|^p_{0,p,\T} &\leq \sum_{\ell \in G} \Delta x \bigg|\frac{1}{\Delta x} \int_{K_\ell} 
	B(y) dy \bigg|^p
  \leq \sum_{\ell \in G} \bigg(\int_{K_\ell} |B(y)|^p dy \bigg) 
	\bigg(\int_{K_\ell} \frac{dx}{\Delta x} \bigg)^{p-1} \\
  &\leq \sum_{\ell \in G} \int_{K_\ell}|B(y)|^p dy = \|B\|^p_{L^p(\torus)},
\end{align*}
which concludes the proof.
\end{proof}

\begin{lemma}\label{lem.aux3}
Let $s >1$ and $p>1$. Then for any sequence $u=(u_\ell)_{\ell \in G}$, there exists a constant $C>0$ only depending on $s$ such that
\begin{align*}
  \|u\|_{0,\infty,\T} \leq C \|u\|_{1,p,\T}^{1/s} \|u\|^{1-1/s}_{0,(s-1)p/(p-1),\T}.
\end{align*}
\end{lemma}

\begin{proof}
We adapt the proof of \cite[Lemma 4.1]{BCF15} to the one-dimensional case. By the embedding $\mathrm{BV}(\torus) \hookrightarrow L^\infty(\torus)$ applied to the sequence $(|u_\ell|^s)_{\ell\in G}$,
\begin{align}\label{a.init}
  \|u\|_{0,\infty,\T}^s \leq C \bigg(\|u\|^s_{0,s,\T} + \sum_{\ell \in G} \big||u_\ell|^s - |u_{\ell+1}|^s\big| \bigg).
\end{align}
Since $s>1$, we have
\begin{align*}
  \sum_{\ell \in G} \big||u_\ell|^s - |u_{\ell+1}|^s\big| 
  \leq s \sum_{\ell \in G} \big(|u_\ell|^{s-1}+|u_{\ell+1}|^{s-1}\big) 
  |u_\ell - u_{\ell + 1}|.
\end{align*}
We apply H\"older's inequality with exponents $p$ and $p/(p-1)$:
\begin{align*}
  \sum_{\ell \in G} \big||u_\ell|^s - |u_{\ell+1}|^s\big| 
  \leq 2 s\bigg(\sum_{\ell \in G} \frac{|u_\ell - u_{\ell+1}|^p}{\Delta x^{p-1}}\bigg)^{1/p} 
  \bigg(\sum_{\ell \in G} \Delta x |u_\ell|^{\frac{(s-1)p}{p-1}} \bigg)^{(p-1)/p}.
\end{align*}
Besides, using again H\"older's inequality (with the same exponents), we find that
\begin{align*}
  \|u\|_{0,s,\T} = \bigg( \sum_{\ell \in G} \Delta x |u_{\ell}| |u_{\ell}|^{s-1} \bigg)^{1/s} \leq \|u\|^{1/s}_{0,p,\T_m} \|u\|^{(s-1)/s}_{0,(s-1)p/(p-1),\T}.
\end{align*}
Then, inserting the last two inequalities into \eqref{a.init} yields the desired result. This concludes the proof of Lemma~\ref{lem.aux3}.
\end{proof}


\section{Counter-example}\label{sec.counter}

We claim that there exist kernels $B^{ij}$, being an indicator function, 
and piecewise constant functions 
$u_1,\ldots,u_n$ such that the positive semi-definiteness condition
$$
  J := \sum_{i,j=1}^n\int_\torus\int_\torus\pi_i a_{ij} B^{ij}(x-y)u_j(y)u_i(x)dydx\ge 0
$$
is {\em not} satisfied. For this statement, we assume that the matrix 
$(\pi_i a_{ij})\in\R^{n\times n}$ is (symmetric and) positive definite.
With the notation of Section \ref{sec.not}, we set
$\Delta x=1/N$ for some even number $N>5$ and choose $r=3\Delta x/2$ as well as the kernels
$$
  B^{ij}(x) = \mathrm{1}_{(-r,r)}(x)\quad\mbox{for }x\in\torus.
$$
Let $u_i=(u_{i,\ell})_{\ell\in G}\in \mathcal{V}_\T$ for $i=1,\ldots,n$. 
Then we can write $J$ as
\begin{equation}\label{app.J}
  J = \sum_{i,j=1}^n\sum_{\ell,\ell'\in G}\pi_i a_{ij}\widehat{M}^{ij}_{\ell,\ell'}
	u_{j,\ell'}u_{i,\ell},
	\quad\mbox{where }\widehat{M}^{ij}_{\ell,\ell'} = \int_{K_\ell}\int_{K_{\ell'}}
	B^{ij}(x-y)dydx.
\end{equation}
A straightforward, but rather tedious computation shows that the matrix 
$\widehat{M}^{ij}=(\widehat{M}^{ij}_{\ell,\ell'})_{\ell,\ell'\in G}$ $\in\R^{N\times N}$
is pentadiagonal with entries
$$
  M_{\ell,\ell'}^{ij} = (\Delta x)^2, \quad M_{\ell,\ell\pm 1}^{ij} = \frac{7}{8}(\Delta x)^2,
	\quad M_{\ell,\ell\pm 2}^{ij} = \frac{1}{8}(\Delta x)^2.
$$
This matrix possesses the eigenvector $w\in\R^N$, defined by $w_\ell=1$ for $\ell$ odd 
and $w_\ell=-1$ for $\ell$ even, associated with the negative 
eigenvalue $\lambda=-4(\Delta x)^2$.

Let $v_1,\ldots,v_n\in\R^n$ be the eigenvectors of the symmetric matrix
$(\pi_i a_{ij})_{i,j=1,\ldots,n}$ associated with the eigenvalues $0<\nu_1\le\ldots\le\nu_n$, respectively.
We define the $nN\times nN$ matrix $\widehat{M}=(\pi_i a_{ij}\widehat{M}^{ij})$
consisting of the $N\times N$ blocks $\pi_i a_{ij}\widehat{M}^{ij}$. It can be verified
that the matrix $\widehat{M}$ possesses the eigenvector $z=(z_1,\ldots,z_n)\in\R^{nN}$ with
$z_i=v_{n,i}w\in\R^N$ for $i=1,\ldots,n$
associated with the eigenvalue $\lambda\nu_n=-4(\Delta x)^2\nu_n$. Then,
choosing $u_i=z_i$ in \eqref{app.J}, we find that
$$
  J = \sum_{i,j=1}^n\pi_i a_{ij}z_i^\top \widehat{M}^{ij}z_j 
	= -4(\Delta x)^2\nu_n\sum_{i=1}^n |z_i|^2 < 0.
$$
This provides the desired counter-example.

\end{appendix}


\end{document}